\documentclass[11pt]{article}
\usepackage{amsmath,amsthm,amsfonts,amssymb,mathrsfs,epsfig,bm}
\usepackage{lipsum}
\usepackage[usenames]{color}
\usepackage[colorlinks=true]{hyperref}
\usepackage{graphicx}
\usepackage[scale=.8]{geometry}
\usepackage{ulem}
\usepackage{textcomp}
\usepackage{amsmath}
\usepackage{esint}
\usepackage{cleveref}
\usepackage[export]{adjustbox} % For images align
\usepackage{caption}
\usepackage{subcaption}
\usepackage{authblk}
\usepackage{mathtools}

\parskip        \medskipamount
\theoremstyle{plain}

\newtheorem{theorem}{Theorem}[section]
\newtheorem{proposition}{Proposition}[section]
\newtheorem{lemma}{Lemma}[section]
\newtheorem{assumption}{Assumption}

\theoremstyle{remark}
\newtheorem{remark}{Remark}
\newtheorem{corollary}{Corollary}
\newtheorem{definition}[theorem]{Definition}
\newtheorem*{example}{Example}

 \def\beqlb{\begin{eqnarray}}\def\eeqlb{\end{eqnarray}}
 \def\beqnn{\begin{eqnarray*}}\def\eeqnn{\end{eqnarray*}}

\def\B{\mathcal{B}}
\def\M{\mathcal{M}}
\def\N{\mathbb{N}}

\def\R{\mathbb{R}}

\def\P{\mathbb{P}}
\def\E{\mathbb{E}}

\renewcommand{\epsilon}{\varepsilon}

\renewcommand{\limsup}{\varlimsup}
\renewcommand{\liminf}{\varliminf}

\newcommand{\rt}{\operatorname{root}}

\definecolor{mygray}{gray}{0.9}
\definecolor{deeppink}{RGB}{255,20,147}
\definecolor{mygreen}{rgb}{0.05, 0.576, 0.03}
\definecolor{myred}{rgb}{0.768, 0.09, 0.09}

\long\def\symbolfootnote[#1]#2{\begingroup
\def\thefootnote{\fnsymbol{footnote}}\footnote[#1]{#2}\endgroup}
\def\TT{\msc{T}}
\def\tt{\mathbf{t}}

\newcommand{\Ind}[1]{\mathbf{1}_{\left\{#1\right\}}}

\renewcommand{\L}{\mcl{L}}
\newcommand{\Lb}{\overline{\L}}
\newcommand{\lab}{\bar{\lambda}}

\newcommand{\mcl}{\mathcal}

\newcommand{\msc}{\mathscr}
\newcommand{\Ll}{\left}
\newcommand{\Rr}{\right}
\newcommand{\bracket}[1]{\left\langle{#1}\right\rangle}
\newcommand{\norm}[1]{\left\Vert{#1}\right\Vert}

\usepackage{MnSymbol}

\usepackage{authblk}
\begin{document}

\title{\bf Perron--Frobenius theorem for a general tree-valued growth-fragmentation-isolation process}
\author{Chenlin Gu\thanks{Yau Mathematical Sciences Center, Tsinghua University}, \, Mingyuan Shen\thanks{Qiuzhen College, Tsinghua University}, \, Ronghang Zhang\thanks{Qiuzhen College, Tsinghua University}}

\maketitle
\begin{abstract}
A general tree-valued dynamics is considered in continuous time:  new vertices are added, and the percolation happens on the links, and the connected components can be frozen. The model is an infinite-type branching process. The main result establishes the Perron--Frobenius type theorem on this model, which extends the previous work \emph{[Ann. Appl. Probab. 33 (6B) 5233 - 5278]}. The proof does not rely on any property of the uniform random recursive tree.
\end{abstract}
 \vspace{8pt} \noindent {\textbf{Key words:} branching process, Perron--Frobenius theory,  recursive tree, Lyapunov function}
 
 
 %\begin{figure}[h!]
 %    \centering
 %    \includegraphics[scale =0.5]{cover.png}
 %    \caption{An illustration of the growth-fragmentation-isolation process with 62 active vertices (in red) and 77 inactive vertices (in blue). }
 %    \label{fig:cover}
 %\end{figure}

\section{Introduction}
\subsection{Main result}
The branching process is a classical probabilistic model and has rich applications in the study of biology, population, computer science, and statistical physics. The object of this paper aims to extend the tree-valued process proposed by Bansaye, Yuan, and the first author in \cite{bansaye2021growth}. This process is called \textbf{the growth-fragmentation-isolation process} (GFI), which was defined to study the contact-tracing in the epidemic. Let us recall at first its original definition. The GFI process starts with a single vertex, and evolves as a process taking value in forest, a collection of trees, as follows:
\begin{enumerate}
	\item[(G)] Every vertex creates a birth of new vertex with rate $\beta$, and a bond is also added to connect the two. Every vertex is labeled by its order (e.g. $1,2,3,4 \cdots$) of birth.
    
	\item[(F)] Every bond is removed independently with a rate $\gamma$, then the fragmentation happens and the tree containing the bond is divided into two trees.
    
	\item[(I)] Every vertex is detected independently with the rate $\theta$, then all the vertices in the same tree are frozen/isolated and stop the evolution.
\end{enumerate}

\begin{figure}[h!]
    \centering
    \includegraphics[width=0.4\linewidth]{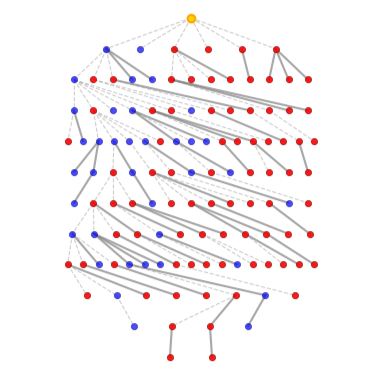}
    \includegraphics[width=0.4\linewidth]{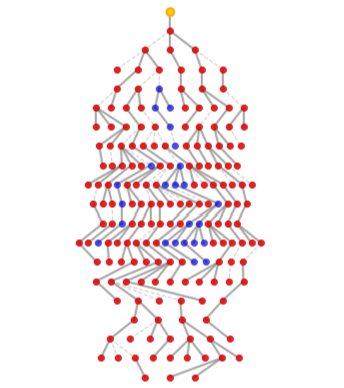}
    \caption{For the parameter setting $(\beta,\gamma,\theta)=(0.6,1,0.1)$, the simulated images illustrate two processes: one that creates a single new vertex each time (left), and one that creates a new chain of size 5 each time (right). The red vertices are active, while the blue vertices are isolated. All edges that have been reomved once are represented by dotted lines. }
\end{figure}

The growth is interpreted as the infection in the epidemic, and the bonds record the information of infection. This information is useful in the contact-tracing and quarantine. However, the bond as the information can be lost in the real world due to various causes. Thus the dynamic of the fragmentation is suggested.  A major feature of the GFI process is the branching both in the macroscopic scale and in the microscopic scale. Among the pandemic, several other models of similar spirit were proposed; see for example \cite{bertoin2022, bellin2023uniform, bellin2024uniformattachmentfreezingscaling, Barlow, Lambert}. The structured growth-fragmentation also appears in the multicellular life-cycle evolution; see \cite{pichugin2017fragmentation, pichugin2020evolution} for example.

The GFI process is Markovian and can be seen as an infinite-type branching process, where each type of tree has different law of production. More precisely, here the tree is equipped with labels, and the order increases along the path from the root to every vertex. This labeled tree is known as \textbf{the recursive tree} / increasing tree in the literature; see Section~\ref{subsec.recursiveTree} for its rigorous definition. Given the size of the recursive denoted by $n$, the uniform distribution over the $(n-1)!$ configurations is called \textbf{the uniform random recursive tree} (RRT), which has various nice properties; see \cite{meir1974cutting, baur2014cutting}. One key idea to study the long-time behavior in the previous work \cite{bansaye2021growth} is also related to this combinatorial object: conditioned on the size, every labeled tree in the GFI process follows the law of RRT. Therefore, the GFI process has a reduced version, where the type in branching is indexed by the size of trees. Afterwards, the law of production is more explicit; see \cite[Section~3.3]{bansaye2021growth} for the details.

Despite the nice combinatorial technique, we ask whether there exists a more robust approach to study the long-time behavior directly on the space of the recursive tree. This question was raised in \cite[Section~8.3]{bansaye2021growth}. Actually, the property of RRT is so nice that it is broken even when a very tiny perturbation is posed. For example, if $2$ vertices are added in every growth in GFI, then we will no longer find RRT in the GFI process. We hope to make GFI process more general, in order to study other related models.

The present work answers the question above. Let $\TT$ stand for the space of the recursive trees. Then we propose the following general growth dynamic.
\begin{enumerate}
	\item[(G')] With the rate $\beta$, every vertex independently has a growth. Once the growth occurs, with probability $p_{\tt}$, a tree $\tt \in \TT$ is added and its root is attached to the vertex with a bond.
\end{enumerate}
The law naturally satisfies $\sum_{\tt \in \TT} p_{\tt} = 1$. For the technical reason, we further assume the moment condition in this paper. Here $\vert \tt \vert$ refers to the number of vertices of $\tt$. It is possible that such moment condition can be further relaxed.
\begin{assumption}\label{assum}
$\sum_{\tt\in\TT}p_{\tt}|\tt|^2<+\infty$.
\end{assumption}

%%%% Old short version %%%%
\iffalse{
Compared to \cite{bansaye2021growth}, the difference is the growth term. If the vertex is attached one by one, then the Yule process is the natural object in the model, and the process can be reduced from $\TT$ to $\N$ using the properties of the recursive tree (RRT). Afterwards, the argument of Lyapunov function help entail a Perron--Frobenius theorem; see \cite[Sections~3.2, 3.3]{bansaye2021growth} for details. However, this is the only case to which the structure of RRT applies, and it breaks immediately otherwise. 

Since we believe the asymptotic behavior is universal, the object of this paper is to develop a more robust method to study this family of processes. This question was also mentioned in \cite[Section~8.3]{bansaye2021growth}.
}\fi

The main result of the paper establishes the Perron--Frobenius theorem for the general GFI process. This theorem is the basis for the long-time behavior of multi-type branching process; see the seminal work \cite{athreya1968some} by Athreya. However, in the infinite-type branching, the Perron--Frobenius theorem is not free. 

We introduce quickly the necessary notations for the main theorem. Let the vector $X_s := (X_s(\tt))_{\tt \in \TT}$ to represent the number of trees
\begin{align}\label{eq.Xt}
    X_s(\tt) := \text{ number of the tree } \tt \text{ at the moment } s.
\end{align}
We also denote by $\P_{\delta_{\tt}}$ the probability space of the process starting from a single tree $\tt$, and $\E_{\delta_{\tt}}$ its associated probability. The first moment semigroup is then defined as 
\begin{align}\label{eq.Mt}
    M_s(\tt, \tt') := \E_{\delta_{\tt}}[X_s(\tt')].
\end{align}
Let $\bracket{ \, , \,}$ stand for the inner product of two vectors. For a positive function $f$ defined on $\TT$, then $M_s f(\tt)$ is defined as 
\begin{align*}
    M_s f(\tt)  := \sum_{\tt' \in \TT} M_s(\tt, \tt') f(\tt') = \E_{\delta_{\tt}}[ \bracket{X_s, f} ].
\end{align*}
In the statement, $\lambda$ stands for the Perron root, and $\pi, h$ are the associated normalized left/right Perron functions.
\begin{theorem}\label{thm.main}
	Under Assumption~\ref{assum}, there exists a unique triplet $(\lambda, \pi, h)$ where $\pi, h : \TT \rightarrow  (0,\infty)$ and $\lambda \in \R$, such that for all $s\geq 0$,
	$$ \pi M_s = e^{\lambda s} \pi, \qquad M_s h =  e^{\lambda s} h,$$
	  and $\sum_{\tt \in \TT}\pi(\tt) = \sum_{\tt \in \TT}  \pi(\tt) h(\tt) =1.$  The vector $\pi$ satisfies that 
    \begin{align}\label{eq.3moment}
        \sum_{\tt \in \TT} \pi(\tt) \vert \tt\vert^2<\infty. 
    \end{align}
    Besides,  there exists $C, \omega > 0$ such that for all $\tt\in \TT$ and $s \geq 0$,
	\begin{align}\label{eq.spectral}
		\sup_{\substack{f:\TT \to \mathbb{R} \\
        \sup_{\tt\in \TT} \vert f(\tt) \vert / \vert \tt \vert^2  \leq 1}}\big\vert e^{-\lambda s} M_s f(\tt) -  h(\tt) \langle\pi,f\rangle \big\vert  &\leq C \vert \tt \vert^{2}  e^{-\omega s}.
	\end{align}
	
	%\comment{V : je mettrais  $\sum_{n\geq 1} \pi(n) n^p <\infty$ a la ligne, pas a cote de la convergence exponentielle, mais je ne veux pas rechanger si vous voulez comme cela ou que c'est appele ensuite avec ce label \eqref{eq:Spectral}}

	%\comment{V: ca me semble le bon endroit de mettre les controles sur les elements propres, remplacerait la remarque  \eqref{pinp} qui arrive pour la strong convergence}
	%\begin{equation}\label{hpi}
	%\sup_{n\geq 1} h(n)  n^{-\delta_0}<\infty, \qquad  
\end{theorem}

Once the Perron--Frobenius theorem is established, we can also derive the limit theorem for the additive martingale.
 \begin{theorem}\label{Thm:a.e.ConvergenceIntro}
    Under Assumption~\ref{assum} and let $(\lambda, \pi, h)$ be the triplet in Theorem~\ref{thm.main}. If $\lambda>0$, then there exists a non-negative random variable $W$, such that for every function satisfying $\sup_{\tt\in \TT} \vert f(\tt) \vert / \vert h(\tt) \vert  <+\infty$, we have 
    \begin{equation}
    \label{eq:a.e.ConvergenceOff}
    \lim_{s\rightarrow\infty}e^{-\lambda s}\langle X_s,f\rangle=\langle \pi,f\rangle W\qquad\text{a.s. and in }L^1.
    \end{equation}
    In addition,  we can take $f \equiv 1$ on $\TT$.
    %For $\lambda\le0$, extinction occurs almost surely in finite time.
 \end{theorem}

\subsection{Related literature}
The spectral analysis of multitype branching processes is classically rooted in the Perron–Frobenius theory for nonnegative matrices and operators. In the finite-type setting, irreducibility of the mean offspring matrix guarantees the existence of a dominant positive eigenvalue together with strictly positive left and right eigenvectors, providing the natural normalization for additive martingales; standard references include  \cite{Harris63} and \cite{AN2004}. Under suitable moment assumptions, notably the Kesten--Stigum $X \log X$ condition, the associated additive martingale converges in $L^1$ to a non-degenerate limit; see \cite{kesten1966limit}.

For branching processes with countably infinite or more general type spaces, the situation is substantially more delicate. In contrast to the finite-dimensional setting, Perron--Frobenius theorem no longer applies automatically, since the spectral radius of the first-moment operator need not correspond to an isolated eigenvalue with positive eigenfunctions. For countable type spaces, a foundational framework was developed by Vere--Jones in \cite{VJ62} and the subsequent series of papers on irreducible countable nonnegative matrices, where recurrence classifications such as transience, null recurrence, and positive recurrence play a central role. Related operator-theoretic perspectives may be found in \cite{Seneta06}, while general branching-process treatments appear in \cite{jagers75}. See also \cite{andre2025sharpllogl, bansaye2025stronglawlargenumbers} for very recent study about the Kesten--Stigum theorem on countable-type branching, assuming the asymptotic behavior of semigroup.

\subsection{Strategy of proof}
The proof of this paper combines two ingredients. The first one is the Lyapunov function, especially the approach developed in \cite{bansaye2019non}. 
\begin{align*}
    \L V \leq a V + \zeta\psi, \qquad b \psi \leq \L \psi  \leq \xi \psi.
\end{align*}
Here $\L$ is the first moment generator; see \eqref{eq.generatorDecom} and \eqref{eq.generatorGFI} for its definition. Such argument has already been applied in the previous work \cite{bansaye2021growth}, while the main challenge in the general GFI process is the abstract fragmentation dynamic. Without the combinatorial technique from RRT, we cannot reduce the branching process. Therefore, the fragmentation dynamic on $\TT$ highly depends on the geometry of the tree, and it is not easy to find $\psi$ satisfying the conditions.

The solution relies on the second ingredient, which is the finite-dimensional approximation. Roughly, we define a process truncated below the cluster size $n$, so the finite-dimensional Perron--Frobenius theorem applies and we obtain a triplet $(\lambda_n, \pi_n, h_n)$
\begin{align*}
    \L_n^* \pi_n = \lambda_n  \pi_n, \qquad \L_n h_n = \lambda_n h_n.
\end{align*}
A detailed and model-dependent analysis then establishes various properties about $(\lambda_n, \pi_n, h_n)$, which holds uniformly for all $n \in \N$. This allows us to derive a limit candidate triplet $(\tilde{\lambda}, \tilde{\pi}, \tilde{h})$ up to subsequence. 

The information of the triplet $(\tilde{\lambda}, \tilde{\pi}, \tilde{h})$ is quite limited \textit{a priori}, but it helps in the construction of Lyapunov function. Then, we can finally confirm that it is indeed the Perron triplet associated to the infinite-type branching.

\section{Preliminaries and notation}
We define rigorously the notations utilized throughout the paper.
\subsection{Recursive tree}\label{subsec.recursiveTree}
Given a set ${V=\{a_1, \cdots, a_n\} \subset \R}$ with increasing order $a_1 < a_2 < \cdots < a_n$, a \textit{recursive tree} $\tt$ on $V$ is a rooted tree labeled by  $V$ such that for any $a_i, 2 \leq i \leq n$, the path  from $a_1$ to $a_i$ is increasing. Thus, the descendants of each vertex have a larger label. The minimal element $a_1$ is called the \textit{root} of $\tt$. The collection of all the recursive trees on $V$ is denoted by $\TT_V$ and it is clear $\vert \TT_V \vert = (\vert V \vert-1)!$.

We also define the equivalence relation $\sim$ between the recursive trees on different ordering sets.  Denoting by $\tt_1$ a recursive tree on $V_1$ and $\tt_2$ a recursive tree on $V_2$, then $\tt_1 \sim \tt_2$ if and only if there exists an order-preserving function $\psi : V_1 \to V_2$, such that $\psi$ is also a bijection  between the graphs $\tt_1$ and $\tt_2$. We denote by $\TT_n$ the set of recursive trees of size $n$ up to the equivalence relation $\sim$, and use the recursive trees defined on $\{1, \cdots, n\}$ as a representative of the equivalent class. Finally, we define the space of finite recursive trees
\begin{align}\label{eq:defRRTWhole}
	\TT := \bigcup_{n=1}^{\infty} \TT_n.
\end{align}  

Finally, given $\tt \in \TT$, we denote by $V(\tt)$ and $E(\tt)$ respectively the set of vertices and the set of bonds/edges, and usually use $\tt$ to refer to $V(\tt)$ as a shorthand notation. Let $\rt(\tt)$ stand for the root of $\tt$, and $\vert \tt \vert$ for the number of vertices in $\tt$. Especially, we denote by $*$ the tree formed by a single vertex.

\subsection{Generator}
We use the language of generator to reformulate rigorously the dynamic of the general GFI process in this paper. 

Given two recursive trees $\tt, \tt' \in \TT$ and $v \in \tt$, we denote by $\tt\oplus_v\tt'$ the tree obtained by ``growth of $\tt'$ at $v$": shift the labels of vertices in $\tt'$ so that they are larger than those in $\tt$, then attach the root of $\tt'$ at $v$ via a bond. See Figure~\ref{fig.growth} for an illustration.
\begin{figure}[h!]
    \centering
    \includegraphics[width=0.7\linewidth]{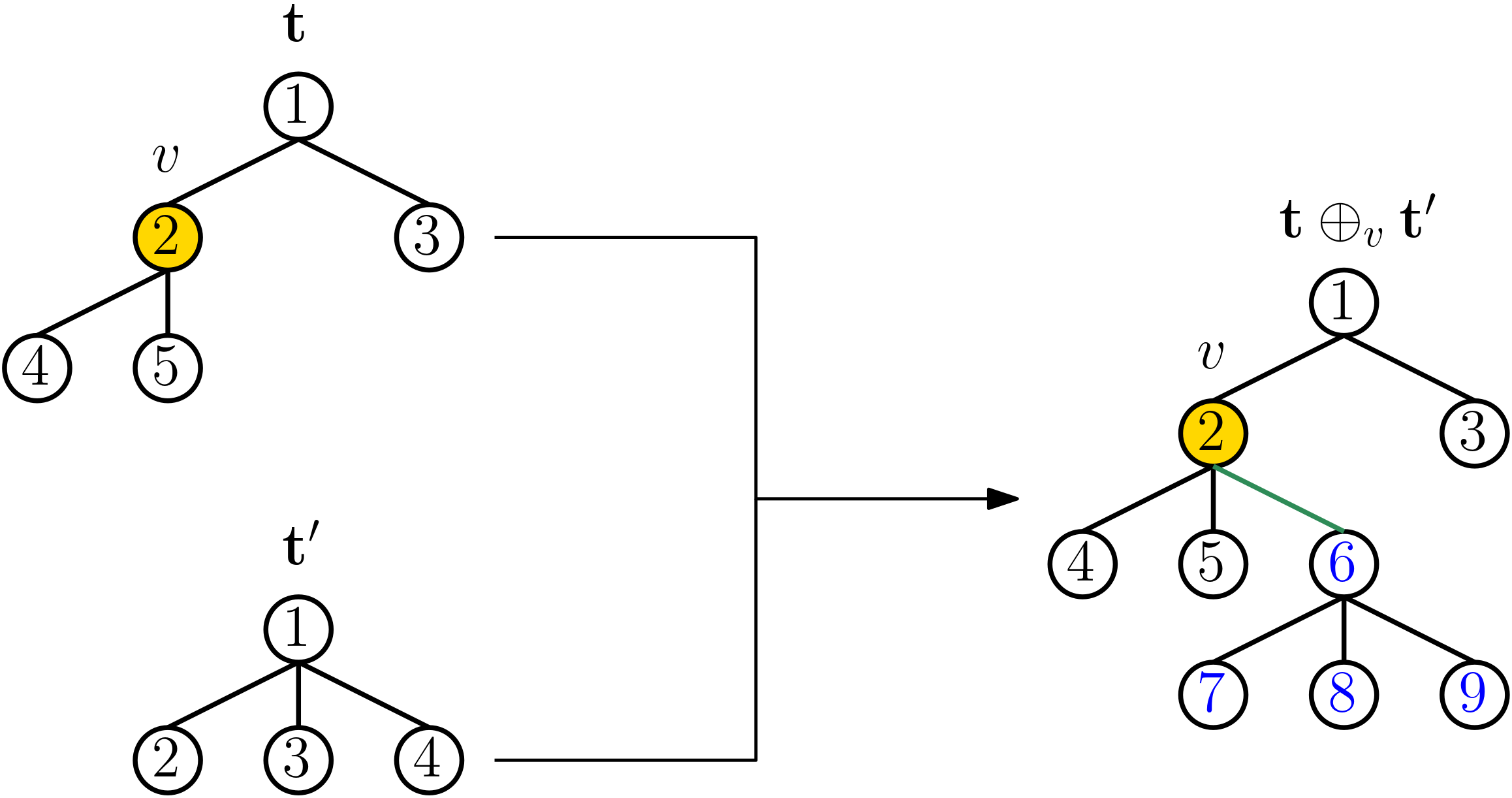}
    \caption{An illustration of the growth dynamic.}
    \label{fig.growth}
\end{figure}

Let $\tt_v$ be the subtree of the recursive tree $\tt$ with root $v$, and let $\tt\ominus\tt_v$ be the tree obtained after removing $\tt_v$ from $\tt$. See Figure~\ref{fig.fragmentation} for an illustration.
\begin{figure}[h!]
    \centering
    \includegraphics[width=0.7\linewidth]{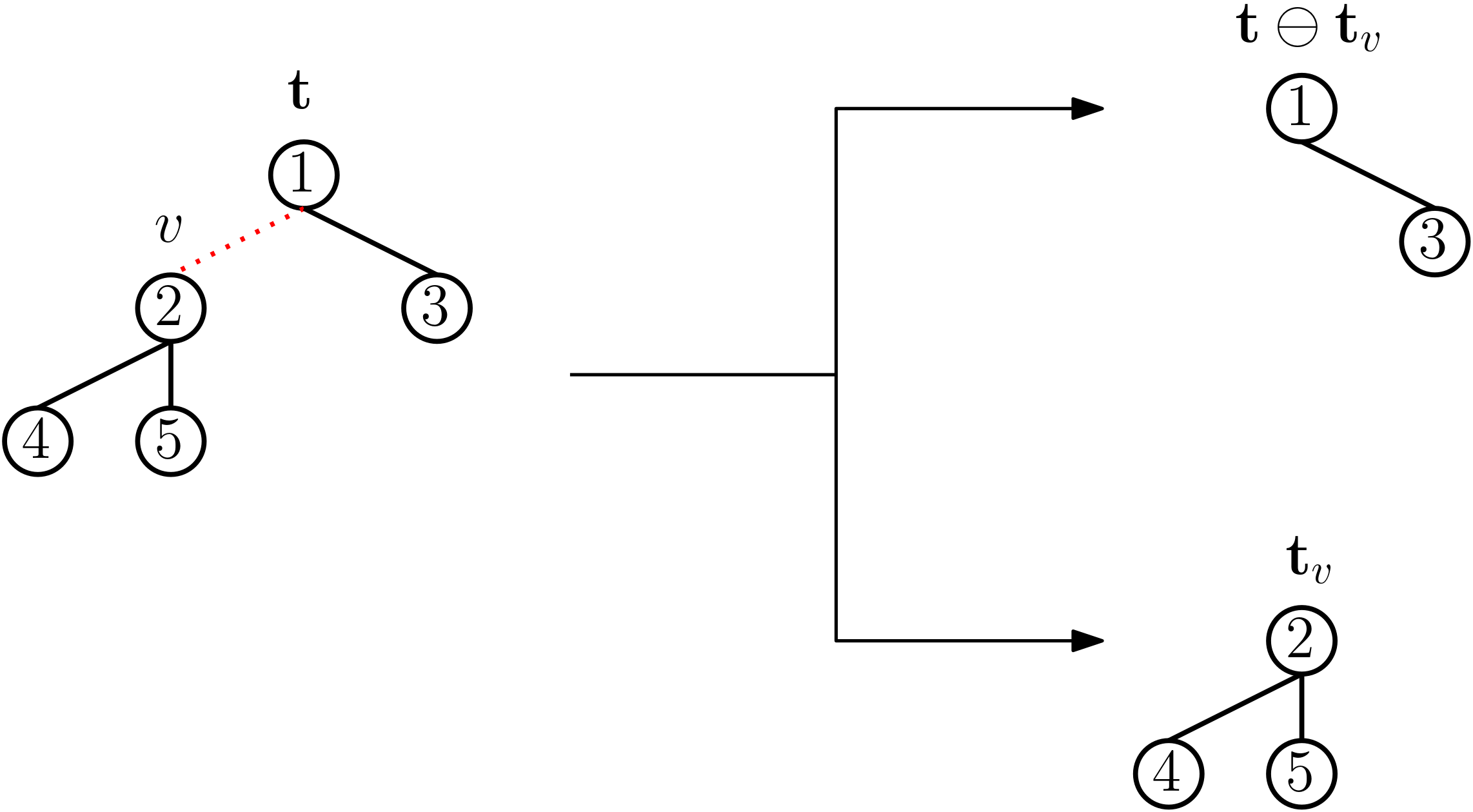}
    \caption{An illustration of the fragmentation dynamic.}
    \label{fig.fragmentation}
\end{figure}
%\begin{align}\label{eq.def.generator}
%    \frac{d}{ds}M_sf(\tt)=M_s(\L f)(\tt).
%\end{align}
The first order generator $\L$ consists of the dynamics of growth, fragmentation, and isolation
\begin{align}\label{eq.generatorDecom}
	\L  = \L_{G}  + \L_{F} + \L_{I}, 
\end{align}
which is well-defined for all bounded function $f$
\begin{equation}\label{eq.generatorGFI}
	\Ll\{\begin{array}{ll}
		\L_{G} f(\tt) &:= \beta  \sum_{v \in \tt}\sum_{\tt' \in \TT}p_{\tt'}(f(\tt \oplus_v \tt') -f(\tt)),\\
		\L_{F} f(\tt) &:= \gamma \sum_{v \in \tt \setminus\{\rt\}} \Ll(f(\tt_v) + f(\tt \ominus \tt_v) - f(\tt)\Rr),\\
		\L_{I} f(\tt) &:= - \theta \vert \tt \vert f(\tt).
	\end{array}\Rr.
\end{equation}
Recall that $p_{\tt'}$ is the probability that $\tt'\in\TT$ is attached in the growth event. In the sum, we just write $\rt$ as a shorthand notation for $\rt(\tt)$, which is clear according to the context.

We write $\L^*$ as its dual operator so that $\bracket{g, \L f} = \bracket{\L^*g, f}$. For convenience, we also identify $\L$ as a matrix $\L(\tt, \tt')_{\tt, \tt' \in \TT}$, so that 
\begin{align*}
    \L f(\tt) = \sum_{\tt' \in \TT} \L(\tt, \tt') f(\tt').
\end{align*}

\subsection{Functions}

  %We denote the corresponding probability distribution by $p$, and for any function $f$ denote the expectation with respect to $p$ by $\E_pf$. 
 
  We denote by the $\mathcal{M}_{F}(\TT)$ the collection of point measures
  $$\mathcal{M}_{F}(\TT):=\Ll\{\text{finite non-negative integer-valued measures on } \TT\Rr\}.$$ 
  Then a forest $\mathbf{F}$ consisting of $\{\tt_i\}_{i \in \N}$ with multiplicity $\{m_i\}_{i \in \N}$ can be identified as $\sum_{i \in \N} m_i\delta_{\tt_i}$, thus an element in $\mathcal{M}_{F}(\TT)$. The GFI process $X_t$ is a Markov jump process on $\mathcal{M}_F(\TT)$. 
  
  %For a measure $\mu$ and a function $f$ on $\TT$, define 
  %$$\langle\mu,f\rangle:=\sum_{\tt\in\TT}\mu(\tt)f(\tt).$$
  
  Let $|\cdot|^\alpha$ represent the function
  $$|\cdot|^\alpha(\tt):=|\tt|^\alpha.$$
  For every $\alpha>0$, we denote the $\alpha$-moment of $\{p_{\tt}\}_{\tt\in\TT}$ by
  $$\mathbf{m}_{\alpha}:=\sum_{\tt\in\TT}p_{\tt}|\tt|^{\alpha}$$
  
  %Let $\Ind{\tt}$ represent the indicator function on $\tt$:
  %$$\Ind{\tt}(\tt')=\delta_{\tt,\tt'}.$$
  
  Given a positive function $g$ on $\TT$, we define the norm 
  \begin{align}\label{eq.def.normsup}
        \norm{f}_{\B(g)} := \sup_{\tt\in\TT} \Ll\vert \frac{f(\tt)}{g(\tt)} \Rr\vert,
  \end{align}
  and we also let $\B(g)$ refer to the functions of finite norm, which is a Banach space.

\section{Modified GFI process and truncated process}

\subsection{Definitions}

We introduce a \textbf{modified GFI process}, which has better properties compared to the original one. The modified GFI process has the same growth and fragmentation dynamics as the original. In the isolation mechanism, each edge is detected independently at rate $\theta$, and the corresponding cluster is then isolated. The associated  first moment generator $\Lb$ of the modified process is:
\begin{align}\label{eq.Lb}
	\Lb f(\tt) = \L f(\tt) + \theta f(\tt).
\end{align}
We then define the truncated modified GFI process, and we just call it \textbf{the truncated process} in short. The generator is defined by 
\begin{align}\label{eq.truncated}
	\Lb_n f(\tt) :=   \beta\Ind{\vert \tt \vert \leq n}\sum_{v\in\tt}\sum_{|\tt'|\le n}p_{\tt'}(f(\tt\oplus_v\tt')-f(\tt))  + \L_{F}f(\tt) + (\L_{I} + \theta)f(\tt).
\end{align} 
That is, when the size of some cluster is greater than $n$, the growth is stopped and only fragmentation and isolation are allowed. In addition, trees larger than $n$ are not allowed to be attached in the growth step. Let $X_{s,n}$ be the corresponding truncated Markov process starting from a single point $*$, then $X_{s,n}$ only concludes clusters of size at most $2n$.  
If the probability of a point to attach in the growth step $p_*$ is greater than $0$, we know that for any $|\tt|\le 2n$ and $s>0$, the probability that $X_{s,n}$ has the cluster $\tt$ is greater than $0$, thus we can view $\Lb_n$ as an irreducible Metzler matrix on $\TT_{2n}$. We denote the $(\tt,\tt')$ element of the matrix by $\Lb_n(\tt,\tt')$, which represents the jump rate from $\tt$ to $\tt'$.  However, the $|\TT_{2n}|\times|\TT_{2n}|$ matrix $\Lb_n$ may not be irreducible in the general case where $p_*>0$ may not hold. In the general case, we define $A_n$ the set of clusters that may appear in $X_{s,n}:$
$$A_n:=\Ll\{\tt\in\TT\big|\ \exists \{\tt_i\}_{i=0}^{m}\text{ s.t.  }\tt_0=*,\tt_m=\tt,\Lb_n(\tt_i,\tt_{i+1})>0\Rr\}\subset\TT_{2n}.$$
Then $\Lb_n$ can be viewed as an irreducible matrix on the finite set $A_n$, since for any $\tt,\tt'\in A_n$, starting from $\tt\in A_n$ we can first reach $*$ through fragmentation and then reach $\tt'$.
Thus, the classical finite-dimensional Perron--Frobenius theorem applies, and we have:
\begin{align}\label{eq.PFn}
	\Lb_n^* \pi_n = \lab_n  \pi_n, \qquad \Lb_n h_n = \lab_n h_n
\end{align}
for some $\lab_n\in\mathbb{R}$ and positive vectors $\pi_n,h_n\in\R_{+} ^{A_n}$. We also view $\pi_n$ and $h_n$ as non-negative vectors on $\TT$ by setting $$\pi_n(\tt)=h_n(\tt)=0,\forall \tt\notin A_n.$$ 
We define the truncated first moment semigroup to be the $|A_n|\times|A_n|$ matrix with $(\tt,\tt')$ element
$$M_{s,n}(\tt,\tt'):=\exp(s\Lb_n)(\tt,\tt')=\E_{\delta_\tt}\Ll[\langle X_{s,n},\Ind{\tt'}\rangle\Rr].$$
Then we have
$$\pi_nM_{s,n}=e^{\lab_n s}\pi_n,\qquad M_{s,n}h_n=e^{\lab_n s}h_n.$$
For appropriate normalization, let $\langle\pi_n,1\rangle =1$ and $h_n(*)=1$ for any n. After studying $h_n$ and $\pi_n$ in section 3.2 and section 3.3,  we give the existence of eigenelements $\lab,h,\pi$ of $\Lb$ as limit of $\lab_n,h_n,\pi_n$ in section 3.4. The following proposition is the goal of the entire section:

\begin{proposition}[Eigenelements of $\Lb$]\label{eigenelements}

	There exist $\lab\in\mathbb{R}$ and positive vectors $h,\pi\in\R_{+}^{\TT}$ such that
	\begin{align}\label{eq.eigenUnique}
		\Lb^* \pi = \lab  \pi, \qquad \Lb h = \lab h,\qquad \langle\pi,1\rangle=\langle\pi,h\rangle=1.
	\end{align}
	Moreover, for any $\alpha>\max\{0,\frac{\gamma-\theta}{\gamma+\theta}\}$, there exist $c,C>0$, such that for every $\tt\in\TT$ we have
    \begin{equation}
    \label{eq:EstimationOfH}
    h(\tt)\geq
    \begin{cases}
    c,&\gamma>\theta\\
    \frac{c}{|\tt|},&\gamma\le\theta
    \end{cases}, 
    \qquad h(\tt)\le C|\tt|^{\alpha},
    \end{equation}
    and
    \begin{equation}
    \label{eq:MomentOfPi}
    \langle\pi,| \, \cdot \,|^2\rangle<+\infty.
    \end{equation}
    In addition, $\lab > 0$ if $\gamma > \theta$, $\lab = 0$ if $\gamma = \theta$, and $\lab < 0$ if $\gamma < \theta$.
\end{proposition}
\begin{remark}
Note that the normalization for $h_n$ is $h_n(*)=1$, but here the normalization for $h$ is $\langle\pi,h\rangle=1$. Thus, $h$ is derived taking the limit of $h_n$ after rescaling.
\end{remark}

\subsection{Estimates of truncated left eigenvector}

    \begin{lemma}
        \begin{enumerate}
        \item $|\lab_n|$ is uniformly bounded in $n$:
            \begin{align}\label{eigenvaluebound}
           \forall n\in\N, \qquad -\beta\leq \lab_n\leq \beta\mathbf{m}_1+\theta. 
            \end{align}
        \item  There exists a constant $C>0$ independent of $n$ such that 
            \begin{align}\label{uniformBoundOfPi_n}
            \forall n\in\N, \qquad  \langle\pi_n, |\cdot|^2\rangle\leq C.
            \end{align}
        \end{enumerate}
    \end{lemma}
    
    \begin{proof}
        We first prove \eqref{eigenvaluebound}. Let $\Ind{*}$ be the indicator function of the single point $*$. It is easy to verify that 
        $$\Lb_n \Ind{*}\ge-\beta\Ind{*}.$$
        Then we have: $$\lab_n\langle\pi_n,\Ind{*}\rangle=\langle\Lb_n^*\pi_n,\Ind{*}\rangle=\langle\pi_n,\Lb_n\Ind{*}\rangle\ge-\beta\langle\pi_n,\Ind{*}\rangle.$$
        Since $\pi_n(*)>0$, we get $\lab_n\ge-\beta$.
        Now we compute
        \begin{align*}
            &\quad\Lb_n|\cdot|(\tt)=\beta\Ind{|\tt|\le n}\sum_{v\in\tt}\sum_{|\tt'|\le n}p_{\tt'}|\tt'|-\theta(|\tt|-1)|\tt|\le-\theta|\tt|^2+(\theta+\beta\mathbf{m}_1)|\tt|.
        \end{align*}
        So we have 
        $$|\cdot|^2\le\frac{(\theta+\beta\mathbf{m}_1)}{\theta}|\cdot|-\frac{1}{\theta}\Lb_n|\cdot|.$$
        Pairing with $\pi_n$ and using $\Lb_n^*\pi_n=\lab_n\pi_n$:
        \begin{equation}
        \label{eq:EstimateOfpi_nf2}
            \langle\pi_n,|\cdot|^2\rangle\le\frac{\theta+\beta\mathbf{m}_1}{\theta}\langle\pi_n,|\cdot|\rangle-\frac{1}{\theta}\langle\pi_n,\Lb_n|\cdot|\rangle=\frac{\theta+\beta\mathbf{m}_1}{\theta}\langle\pi_n,|\cdot|\rangle-\frac{\lab_n}{\theta}\langle\pi_n,|\cdot|\rangle.
        \end{equation}
         Then we use $\langle\pi_n,|\cdot|^2\rangle>0$ to derive $$\lab_n\le\theta+\beta\mathbf{m}_1.$$
         Then, we prove \eqref{uniformBoundOfPi_n}. Using \eqref{eq:EstimateOfpi_nf2} and $\lab_n\ge-\beta$, we have
         $$\langle\pi_n,|\cdot|^2\rangle\le\frac{\theta+\beta+\beta\mathbf{m}_1}{\theta}\langle\pi_n,|\cdot|\rangle.$$
         Apply Cauchy's inequality:
         $$\langle\pi_n,|\cdot|^2\rangle\le\frac{\theta+\beta+\beta\mathbf{m}_1}{\theta}\langle\pi_n,|\cdot|^2\rangle^{\frac{1}{2}}.$$
         This gives \eqref{uniformBoundOfPi_n} by choosing $C=(\frac{\theta+\beta+\beta\mathbf{m}_1}{\theta})^2$ and completes the proof.
    \end{proof}

\subsection{Estimates of truncated right eigenvector} 
This subsection aims to derive the following estimates:
\begin{lemma}\label{lem:eigenvectorUpperbound}    
     For any $\alpha>\max\Ll\{0,\frac{\gamma-\theta}{\gamma+\theta}\Rr\}$, there exists $C>0$ such that  for every $n\in\N$,$\tt\in A_n$, we have
     \begin{align*}
         h_n(\tt)\le C|\tt|^{\alpha}.
     \end{align*}  
     \end{lemma}
\begin{lemma}\label{lem:eigenvectorLowerbound}
    There exists $c>0$ such that for every n, $\tt\in A_n$ we have $$h_n(\tt)\geq\frac{c}{|\tt|}.$$
    When $\gamma>\theta$, there exists $c>0$ such that for every n, $\tt\in A_n$ we have $$h_n(\tt)\ge c.$$
\end{lemma}

\subsubsection{Upper bound}
    We first prove that $A_n$ is sufficiently large for all large $n$:
    \begin{lemma}
    \label{lem:A_nLarge}
        For any $m\in\N$, there exists $N$ such that for every $n\ge N$, we have
        $$\TT_m\subset A_n.$$
    \end{lemma}
    \begin{proof}
        We take $\tt'\in\TT$ such that $p_{\tt'}>0$ and choose $$N=m+|\tt'|.$$
        For every $|\tt|\le m$, take $k=|\tt|-1$, there exists a sequence of trees $\{\tt_i\}_{i=1}^{k}$ such that $$\tt_{0}=*,\tt_{k}=\tt,$$ and $\tt_{i+1}$ is obtained by attaching a single point on the vertex $v_{i}$ of $\tt_i$. We let 
        $$\tt_{i+\frac{1}{2}}=\tt_{i}\oplus_{v_i}\tt'.$$
        Since $|\tt'|\le n,\  |\tt_i|\le n, \ p_{\tt'}>0\ $, we have 
        $$\Lb_n(\tt_i,\tt_{i+\frac{1}{2}})>0.$$
        And by fragmentation, we also have
        $$\Lb_n(\tt_{i+\frac{1}{2}},\tt_{i+1})>0.$$
        Thus, $\tt_k$ is possible to appear in the process $X_{s,n}$ and we have $\tt=\tt_k\in A_n$.
    \end{proof}
    \begin{corollary}\label{Cor:M_{t,n}LowerBound}
        For every $\tt\in \TT$, there exists $N>0$ such that for all $s_0>0$ we have
        $$\inf_{n\ge N}M_{s_0,n}(*,\tt)>0.$$
    \end{corollary}
    \begin{proof}
        We use the same definition of $N, \tt_i,\tt_{i+\frac{1}{2}}$ as in the proof of Lemma~\ref{lem:A_nLarge}. For every $n>N$, the probability that $X_{t,n}$ has the cluster $\tt$ is bounded below by the probability that the evolution of $X_{s,n}$ before $s_0$ is exactly $$\tt_{0},\tt_{\frac{1}{2}},\tt_{1},...,\tt_{m}.$$
        The latter event is independent of $n$ as long as $n>N$, which completes the proof.
    \end{proof}
    Then we prove a weak upper bound for $h_n$:
    \begin{lemma}
    \label{lem:PointUpperOfh_n}
    Given $\tt \in \TT$, there exists a constant $C_{\tt}$ independent of $n$ such that $$\forall n \in \N, \qquad h_n(\tt)\le C_{\tt}.$$
    \end{lemma}
    \begin{proof}
    By Corollary~\ref{Cor:M_{t,n}LowerBound} we can choose $N$ such that 
    $$\inf_{n\ge N}M_{1,n}(*,\tt)>0.$$
    Then for any $n>N$ we use
    \begin{align*}
        e^{\lab_n}=e^{\lab_n}h_n(*)=M_{1,n}h_n(*)=\sum_{\tt'\in A_n}M_{1,n}(*,\tt')h_n(\tt')\ge M_{1,n}(*,\tt)h_n(\tt)
    \end{align*}
    and the upper bound of $\lab_n$ in \eqref{eigenvaluebound} to obtain
    $$h_n(\tt)\le\frac{e^{\lab_n}}{M_{1,n}(*,\tt)}\le\frac{e^{\sup_{n}\lab_n}}{\inf_{n\ge N}M_{1,n}(*,\tt)}.$$
    This completes the proof since $\frac{e^{\sup_{n}\lab_n}}{\inf_{n\ge N}M_{1,n}(*,\tt)}<+\infty$.
    \end{proof}

    To give a precise upper bound, the following estimate is needed:
    \begin{lemma}
    \label{Lem:TreeEstimate}
    Let \(0<\alpha\le 1\), and let \(\tt \in \TT\). For every edge \(e\in E(\tt)\), deleting \(e\) gives two connected components with sizes \(t_{e,1}\) and \(t_{e,2}\). Define 
    $$F_\alpha(\tt):=\sum_{e\in E(\tt)}\left(t_{e,1}^{\alpha}+t_{e,2}^{\alpha}\right).$$ 
    Then it satisfies 
    $$F_\alpha(\tt)\le 2\sum_{i=1}^{|\tt|-1} i^\alpha.$$
    \end{lemma}

    \begin{proof}
    We first prove the following claim: Let \(R\) be a rooted tree with root \(r\) and \(m\) vertices. For every edge \(e\), let \(x_e\) be the size of the component containing \(r\) after deleting \(e\). If \(x_1\le x_2\le \cdots \le x_{m-1}\) is the increasing ordering of these numbers, then \(x_j\ge j\) for every \(1\le j\le m-1\).

    For each non-root vertex \(v\), let \(s(v)\) be the size of the rooted subtree with root \(v\). If \(e\) is the edge joining \(v\) to its parent, then $$x_e=m-s(v).$$ We show that for every positive integer \(q\), the number of non-root vertices \(v\) satisfying \(s(v)\ge q\) is at most \(m-q\). This is proved by induction on \(m\). Suppose the root \(r\) has child-subtrees of sizes \(m_1,\ldots,m_d\), where $$\sum_{i=1}^d m_i=m-1.$$ A branch with \(m_i<q\) contributes no such vertex. If \(m_i\ge q\), then the root of this branch contributes one such vertex, and by induction the remaining part contributes at most \(m_i-q\) such vertices. Hence, this branch contributes at most \(m_i-q+1\). If \(k\) branches have size at least \(q\), then
    \begin{align*}
    \#\{v\ne r:s(v)\ge q\}&\le \sum_{m_i\ge q}(m_i-q+1)=\sum_{m_i\ge q}m_i-k(q-1)\\
    &\le (m-1)-(q-1)=m-q.
    \end{align*}
    Thus, if \(s_1\ge s_2\ge \cdots \ge s_{m-1}\) is the decreasing ordering of the values \(s(v)\), then \(s_j\le m-j\). Since \(x_e=m-s(v)\), we get \(x_j=m-s_j\ge j\). 

    Now we prove the lemma by induction on \(|\tt|\). The case \(|\tt|=1\) is trivial. Choose a leaf \(z\) of \(\tt\), let \(r\) be its unique neighbor, and set \(\tt'=\tt-z\). Then \(\tt'\) has \(m=|\tt|-1\) vertices. Root \(\tt'\) at \(r\).

    For an edge \(e\in E(\tt')\), let \(x_e\) be the size of the component containing \(r\) after deleting \(e\) in \(\tt'\). The other component has size \(m-x_e\). Thus the contribution of \(e\) to \(F_\alpha(\tt')\) is \(x_e^\alpha+(m-x_e)^\alpha\). After adding \(z\) back, this contribution becomes \((x_e+1)^\alpha+(m-x_e)^\alpha\). Hence the increase caused by \(e\) is \((x_e+1)^\alpha-x_e^\alpha\). Since \(0<\alpha\le 1\), the function \(x^\alpha\) is concave, so \((x+1)^\alpha-x^\alpha\) is non-increasing in \(x\). By the claim, after ordering the \(x_e\)'s increasingly, we have \(x_j\ge j\). Therefore
    \begin{align*}
    \sum_{e\in E(\tt')}\left((x_e+1)^\alpha-x_e^\alpha\right)&\le \sum_{j=1}^{m-1}\left((j+1)^\alpha-j^\alpha\right)=m^\alpha-1.
    \end{align*}
    The new leaf edge contributes \(1+m^\alpha\). Hence $$F_\alpha(\tt)-F_\alpha(\tt')\le (m^\alpha-1)+(1+m^\alpha)=2m^\alpha.$$ By induction, $$F_\alpha(\tt')\le 2\sum_{i=1}^{m-1}i^\alpha.$$ Therefore $$F_\alpha(\tt)\le 2\sum_{i=1}^{m-1}i^\alpha+2m^\alpha=2\sum_{i=1}^{|\tt|-1}i^\alpha.$$ This proves the lemma.
    \end{proof}
Now we prove Lemma~\ref{lem:eigenvectorUpperbound}.
\begin{proof}[Proof of Lemma~\ref{lem:eigenvectorUpperbound}]
    We only need to prove for $\alpha\in\Ll(\max\Ll\{0,\frac{\gamma-\theta}{\gamma+\theta}\Rr\},1\Rr)$. Let $V(\tt)=|\tt|^{\alpha}$. Fix $n$, let 
    $$R_n(\tt)=\frac{h_n(\tt)}{V(\tt)},\qquad\tt_0=\arg \max_{\tt\in A_n}\{R_n(\tt)\}.$$
    The notation $\tt_0$ ignores the dependence on $n$ for simplicity, but the dependence should be remembered. Note that
    $$\Lb_n(\tt,\tt)<0,\qquad\Lb_n(\tt,\tt')\ge 0 ,\tt\neq \tt'.$$
    Then we use $\Lb_nh_n=\lab_nh_n$:
    \begin{align*}
        (\lab_n-\Lb_n(\tt_0,\tt_0))h_n(\tt_0)&=\sum_{\tt\neq\tt_0}\Lb_n(\tt_0,\tt)h_n(\tt)
        \le\sum_{t\neq\tt_0}\Lb_n(\tt_0,\tt)R_n(\tt_0)V(\tt)
    \end{align*}
    combining with $$h_n(\tt_0)=R_n(\tt_0)V(\tt_0),\qquad R_n(\tt_0)>0$$ to get
    $$(\lab_n-\Lb_n(\tt_0,\tt_0))V(\tt_0)\le\sum_{\tt\neq\tt_0}\Lb_n(\tt_0,\tt)V(\tt).$$
    That is,
    \begin{equation}\label{eq:labVlessLV}\lab_nV(\tt_0)\le\Lb_nV(\tt_0).\end{equation} 
    On the other hand, using Lemma~\ref{Lem:TreeEstimate}, we have:
    \begin{align*}
        \Lb_n V(\tt_0) &=\beta\Ind{|\tt_0|\le n}\sum_{v\in\tt_0}\sum_{\tt'\in\TT}p_{\tt'}\Ind{|\tt'|\le n}(V(\tt_0\oplus_v\tt')-V(\tt_0))\\
        &\qquad +\gamma\sum_{v\in\tt_0\setminus\{\rt\}}(V(\tt_v)+V(\tt_0\ominus\tt_v)-V(\tt_0))-\theta(|\tt_0|-1)V(\tt_0)\\       &\le\beta\sum_{v\in\tt_0}\sum_{\tt'\in\TT}p_{\tt'}|\tt'|^{\alpha}+2\gamma\sum_{i=0}^{|\tt_0|-1}i^{\alpha}-(\gamma+\theta)(|\tt_0|-1)|\tt_0|^{\alpha}\\
        &\le \Ll(\frac{2\gamma}{\alpha+1}-\theta-\gamma\Rr)|\tt_0|^{\alpha+1}+C'|\tt_0|.
    \end{align*}
    $C'>0$ is some constant independent of $n,|\tt|$. Using $\lab_n\ge-\beta$ in \eqref{eigenvaluebound} and \eqref{eq:labVlessLV}, we have
    $$-\beta V(\tt_0)\le \Lb_n V(\tt_0)\le\Ll(\frac{2\gamma}{\alpha+1}-\theta-\gamma\Rr)|\tt_0|^{\alpha+1}+C'|\tt_0|.$$
    Since $\alpha>\frac{\gamma-\theta}{\gamma+\theta}$, we have $\frac{2\gamma}{\alpha+1}-\theta-\gamma<0$. Thus, there exists a constant $M$ such that 
    $$|\tt_0|\le M,\forall n$$
    Thus, for all $n\in\N$ and $\tt\in\TT$, applying Lemma~\ref{lem:PointUpperOfh_n} gives
    $$h_n(\tt)\le R_n(\tt_0)V(\tt)=\frac{h_n(\tt_0)}{V(\tt_0)}V(\tt)\le\frac{C_{\tt_0}}{|\tt_0|^{\alpha}}V(\tt)\le \Ll(\sup_{|\tt'|\le M}\frac{C_{\tt'}}{|\tt'|^{\alpha}}\Rr)V(\tt).$$
    This completes the proof.
    \end{proof}

\subsubsection{Lower bound}
    In this part, we will give a uniform lower bound of $h_n$ in Lemma~\ref{lem:eigenvectorLowerbound}. We define 
    $$\tau^{(n)}:=\inf\limits\{ s \in \R_+ : \langle X_{s,n},\delta_{*}\rangle>0\}$$ 
    to be the first time an isolated vertex appears in $X_{s,n}$.
     We have the following estimate:
    \begin{lemma}
    \label{lem:IsolatedVertexAppear}
    When $\gamma>\theta$, for every $n, s\geq 0$, $\ \tt\in A_n$, we have: 
    $$\mathbb{P}_{\delta_{\tt}}(\tau^{(n)}\leq s)\geq \min \Ll\{y(s), 1-\frac{\theta}{\gamma}\Rr\}, $$
    where $y(s)=\frac{\gamma}{\gamma+\theta}\Ll(1-e^{-(\gamma+\theta)s}\Rr).$
    \end{lemma}

    \begin{proof}
    We use the following shorthand notation in the proof
    $$q_s(\tt):=\mathbb{P}_{\tt}(\tau^{(n)}\leq s),$$
    where the dependence on $n$ is omitted. Let $y_{\epsilon}(s)=y(s)-\epsilon e^s$, then it suffices to fix $\epsilon$ and show that $q_s(\tt)\geq \min \Ll\{y_{\epsilon}(s), 1-\frac{\theta}{\gamma}\Rr\}$. 
    We prove it by contradiction. Let $$s_0:=\inf_s\Ll\{\exists \tt\in A_n \text{ s.t. } q_s(\tt)< y_{\epsilon}(s),\ q_s(\tt)< 1-\frac{\theta}{\gamma} \Rr\}.$$
    From $q_0(\tt)\geq 0>y_{\epsilon}(0)$ we imply that $s_0>0$. 
    Take 
    $$\tt_0=\arg\min\limits_{\{\tt\in A_n\}}q_{s_0}(\tt).$$
     Let
     $$\alpha=\Ind{|\tt_0|\le n}\beta|\tt_0|\sum_{|\tt'|\le n}p_{\tt'}+(|\tt_0|-1)(\gamma+\theta)$$ be the total rate of evolution for $\tt_0$, then we prove: 
    \begin{equation}
    q_s(\tt_0)%&=\int_0^s \alpha e^{-\alpha r}(\frac{\beta}{\alpha}\textbf{1}_{\{|\tt|\leq n\}}\sum_{v\in\tt}\sum_{\tt'\in\TT}p_{\tt'}q_{s-r}(\tt\oplus_v\tt')+\frac{\gamma}{\alpha}\sum_{v\in t\setminus\{\rt\}}Q(s-r, \tt,v))\\
    =e^{-\alpha s}\int_0^s e^{\alpha r}\left(\beta\textbf{1}_{\{|\tt_0|\leq n\}}\sum_{v\in\tt_0}\sum_{|\tt'|\le n}p_{\tt'}q_{r}(\tt_0\oplus_v\tt')+\gamma\sum_{v\in \tt_0\setminus\{\rt\}}Q_r(v)\right)dr
    \label{eq:DEofQ}
    \end{equation}
    where 
    $$Q_s(v)=1-(1-q_s(\tt_0\ominus\tt_v))(1-q_s(\tt_v)). $$
    This equation is derived by studying the first event of the truncated process $X_{s,n}$ starting from $\tt_0$:
    The probability density that the first event occurs at time $r$ is $\alpha e^{-\alpha r}$. When the first event occurs, the probability that the event is attaching $\tt'$ at $v\in\tt_0$ is $\frac{\beta p_{\tt'}}{\alpha}\Ind{|\tt_0|\le n,|\tt'|\le n}$. After that, the probability that the isolation point occurs in the rest $s-r$ time is $$q_{s-r}(\tt_0\oplus_v\tt').$$ The probability that the first event is fragmentation at the edge connecting $v$ and its parent is $\frac{\gamma}{\alpha}$. Then after fragmentation, the two remaining clusters $\tt_v$ and $\tt_0\ominus\tt_v$ evolve independently. The probability that an isolated vertex occurs in the rest $s-r$ time is just $$Q_{s-r}(v)=1-(1-q_{s-r}(\tt_0\ominus\tt_v))(1-q_{s-r}(\tt_v)).$$
    If the first event is isolation, it is impossible for an isolated vertex to occur. Thus, we have the following equation:
    $$q_{s}(\tt_{0})=\int_{0}^{s}\alpha e^{-\alpha r}\left(\sum_{v\in\tt_0}\sum_{|\tt'|\le n}\frac{\beta}{\alpha}p_{\tt'}\Ind{|\tt_0|\le n}q_{s-r}(\tt_0\oplus_v\tt')+\sum_{v\in\tt_0\setminus\{\rt\}}\frac{\gamma}{\alpha} Q_{s-r}(v)\right)dr.$$
    Then changing the variable $r\rightarrow s-r$ gives equation \eqref{eq:DEofQ}.
    Differentiating with respect to $s$: 
    \begin{align*}
    \frac{d}{ds}q_s(\tt_0)&=-\alpha q_s(\tt_0)+e^{-\alpha s}\cdot e^{\alpha s}\left(\beta\textbf{1}_{\{|\tt_0|\leq n\}}\sum_{v\in\tt_0}\sum_{|\tt'|\le n}p_{\tt'}q_{s}(\tt_0\oplus_v\tt')+\gamma\sum_{v\in \tt_0\setminus\{\rt\}}Q_{s}(v)\right)\\
    &=\beta\textbf{1}_{\{|\tt_0|\leq n\}}\sum_{v\in\tt_0}\sum_{|\tt'|\le n}p_{\tt'}(q_{s}(\tt_0\oplus_v\tt')-q_s(\tt_0))+\sum_{v\in \tt_0\setminus\{\rt\}}\left(\gamma Q_{s}(v)-\gamma q_{s}(\tt_0)-\theta q_{s}(\tt_0)\right). 
    \end{align*}
    Then we use $q_{s_0}(\tt_0)\le q_{s_0}(\tt),\forall\tt\in A_n$ to derive:
    \begin{align*}
    \frac{d}{ds}\bigg|_{s=s_0}q_{s}(\tt_0)&\geq \sum_{v\in \tt_0\setminus\{\rt\}}(\gamma Q_{s_0}(v)-\gamma q_{s_0}(\tt_0)-\theta q_{s_0}(\tt_0))
    \end{align*}
    Since $\tt_0=\arg\min\limits_{\{\tt\in A_n\}}q_{s_0}(\tt)$, from the definition of $q$ we know that $\tt_0$ is not an isolated vertex. We take any leaf $v_0$ of $\tt_0$ and use $Q_{s}(v_0)=1$ to derive:
    \begin{align*}
    \frac{d}{ds}\bigg|_{s=s_0}q_{s}(\tt_0)&\geq \gamma-(\gamma+\theta)q_{s_0}(\tt_0)+\sum_{v\in\tt_0\setminus\{\rt,v_0\}}(\gamma Q_{s_0}(v)-\gamma q_{s_0}(\tt_0)-\theta q_{s_0}(\tt_0))
    \end{align*}
    Then we use $Q_{s_0}(v)=1-(1-q_{s_0}(\tt_0\ominus\tt_v))(1-q_{s_0}(\tt_v))\geq1-(1-q_{s_0}(\tt_0))^2$ to derive:
    \begin{align*}
    \frac{d}{ds}\bigg|_{s=s_0}q_{s}(\tt_0)&\geq \gamma-(\gamma+\theta)q_{s_0}(\tt_0)+\sum_{v\in\tt_0\setminus\{\rt,v_0\}}\left(\gamma(1-(1-q_{s_0}(\tt_0))^2)-(\gamma+\theta)q_{s_0}(\tt_0)\right)
    \end{align*}
    Since $q_{s_0}(\tt_0)\le 1-\frac{\theta}{\gamma}$, we have $\gamma(1-(1-q_{s_0}(\tt_0))^2)-(\gamma+\theta)q_{s_0}(\tt_0)\ge 0$. Since $q_{s}(\tt_0)$ is continuous, by the contradiction hypothesis we have $q_{s_0}(\tt_0)\le y_{\epsilon}(s_0)$. Thus:
    \begin{align*}
    \frac{d}{ds}\bigg|_{s=s_0}q_{s}(\tt_0)\geq \gamma-(\gamma+\theta)q_{s_0}(\tt_0)\ge\gamma-(\gamma+\theta)y_{\epsilon}(s_0)>y_{\epsilon}'(s_0). 
    \end{align*}
    So, there exists $0<s'<s_0$ such that $q_{s'}(\tt_0)<y_{\epsilon}(s')$. Note that $q_s(\tt_0)$ is increasing with respect to $s$, hence $q_{s'}(\tt_0)<1-\frac{\theta}{\gamma}$. These contradict the assumption. 
    \end{proof}

    Now we prove Lemma~\ref{lem:eigenvectorLowerbound}.
    \begin{proof}[Proof of Lemma~\ref{lem:eigenvectorLowerbound}]
    
    For any $\tt\in A_n$, using $\lab_n\le M$ in \eqref{eigenvaluebound}, we have
    \begin{align*}
        &\quad h_n(\tt)=e^{-\lab_n}M_{1,n}h_n(\tt)\ge e^{-M}M_{1,n}(\tt,*)h_n(*)\\
        &\geq e^{-M}h_n(*)\mathbb{P}_{\tt}(\tau^{(n)}\leq1,\text{ the isolated vertex stays unchanged in the remaining time})\\
        &\ge e^{-M}e^{-\beta}\P(\tau^{(n)}\le 1).   
    \end{align*}
    When $\gamma>\theta$, Lemma~\ref{lem:IsolatedVertexAppear} above gives us the uniform lower bound. When $\gamma\le\theta$, the probability that an isolated vertex occurs before $1$ is obviously lower bounded by the probability that the first event happens before $1$ and the event is fragmentation on an edge between a leaf and its parent, which is lower bounded by $\frac{c}{|\tt|}$ for some $c>0$. Hence the proof is completed. 
    \end{proof}

\subsection{Limit of truncated eigenelements}
In this section, we prove the Proposition~\ref{eigenelements}. From \eqref{eigenvaluebound} we know that $\lab_n$ converges to some limit $\lab$ up to a subsequence. From Lemma~\ref{lem:eigenvectorLowerbound} and Lemma~\ref{lem:eigenvectorUpperbound} we know that $h_n$ converges pointwise to some positive limit $\tilde{h}\in\R_{+}^{\TT}$ up to a subsequence. From \eqref{uniformBoundOfPi_n} and Prokhorov's theorem we know that $\{\pi_n\}$ as a sequence of probability measures on $\TT$ converges weakly to a  probability measure $\pi$ up to a subsequence. Up to the subsequence of a subsequence, with a little abuse of notation, we can further assume that
$$\lim_{n\rightarrow\infty}\pi_n=\pi,\qquad\lim_{n\rightarrow\infty}h_n=\tilde{h},\qquad\lim_{n\rightarrow\infty}\lab_n=\lab.$$
The first two convergences hold separately in the weak and pointwise sense.

    The following lemma characterizes the limit of the left eigenvector.
    \begin{lemma}\label{lefteigenvector}
        The limiting distribution $\pi$ satisfies
        \begin{align}
        \label{eq:PropertyOfPi}
            \Lb^*\pi=\lab\pi,\qquad \langle\pi,|\cdot|^2\rangle<+\infty,
        \end{align}
        and $\pi(\tt) > 0$ for all $\tt \in \TT$.
    \end{lemma}
    \begin{proof}
         By Fatou's lemma, we have     $$\langle\pi,|\cdot|^2\rangle\leq\liminf_{n\rightarrow\infty}\langle\pi_n,|\cdot|^2\rangle\leq C.$$ 
        Thus, $\pi$ has the second moment. To prove $\Lb^*\pi=\lab\pi$, we only need to prove that for every finite supported test function $f$, we have $\langle\Lb^*\pi,f\rangle=\langle\lab\pi,f\rangle$, i.e. $\langle\pi,\Lb f\rangle=\langle\lab\pi,f\rangle$. We have
        \begin{align}
        \label{eq:DifferenceOfPi}
            |\langle\lab\pi,f\rangle - \langle\pi,\Lb f\rangle|\leq\limsup_{n\rightarrow\infty}\left(\underbrace{|\langle\lab\pi,f\rangle-\langle\lab_n\pi_n,f\rangle|}_{\mathbf{I}} + \underbrace{|\langle\pi_n,\Lb_nf\rangle-\langle\pi,\Lb f\rangle|}_{\mathbf{II}} \right).   
        \end{align}
        Since $f$ is bounded, $\langle\pi_n,f\rangle\xrightarrow{n\rightarrow\infty}\langle\pi,f\rangle$. Combining with $\lab_n\rightarrow\lab$, we know that $\mathbf{I}$ in \eqref{eq:DifferenceOfPi} goes to $0$. 
        
        Now we prove that $\mathbf{II}$ in \eqref{eq:DifferenceOfPi} goes to $0$. We fix $K$, and from Lemma~\ref{lem:A_nLarge} we can take $n$ large enough such that the support of $f$ is contained in $A_n$, and $\TT_K\subset A_n$. Then
        \begin{align*}
        \mathbf{II}&\le\underbrace{\sum_{|\tt|\le K}|\pi_n(\tt)\Lb_nf(\tt)-\pi(\tt)\Lb f(\tt)|}_{\mathbf{II}.1}+\underbrace{\sum_{|\tt|>K}|\pi_n(\tt)\Lb_nf(\tt)-\pi(\tt)\Lb f(\tt)|}_{\mathbf{II}.2}.
        \end{align*}
        
        We first look at $\mathbf{II}.1$. For $\tt\in A_n$,  using $|f|\le M$, we have
        $$|\Lb f(\tt)-\Lb_nf(\tt)|\le\beta\sum_{v\in\tt}\sum_{|\tt'|> n}p_{\tt'}|f(\tt\oplus_v\tt')-f(\tt)|\le2\beta M\sum_{v\in\tt}\sum_{|\tt'|>n}p_{\tt'}\rightarrow 0.$$
        Combining with $\pi_n(\tt)\xrightarrow{n \to \infty}\pi(\tt)$, we derive that $\mathbf{II}.1$ goes to $0$ as $n$ goes to infinity, since $K$ is fixed and $\mathbf{II}.1$ is finite sum. 
        
        Now we study $\mathbf{II}.2$. Since f is bounded, it is easy to verify that $\max\Ll\{\Ll|\Lb_nf\Rr|,\Ll|\Lb f\Rr|\Rr\}\leq C_f |\cdot|$ for some constant $C_f$. Then
        \begin{align*}
            \sum_{|\tt|>K}|\pi_n(\tt)\Lb_nf(\tt)-\pi(\tt)\Lb f(\tt)| &\le\sum_{|\tt|>K}C_{f}(\pi_n(\tt)+\pi(\tt))|\tt|\\
            &\le\frac{C_f}{K}\sum_{\tt \in \TT}\Ll(\pi_n(\tt)+\pi(\tt)\Rr)|\tt|^2\\
            &\le\frac{2C_fC}{K}.
        \end{align*}
        We let $n\rightarrow\infty$ and derive that
        $$\limsup_{n\rightarrow\infty}|\langle\pi_n,\Lb_nf\rangle-\langle\pi,\Lb f\rangle|\le\frac{2C_fC}{K}.$$ 
        Finally, we let $K\rightarrow\infty$ and derive \eqref{eq:PropertyOfPi}. 
        
        Now we prove that $\pi$ is positive.
    Assume $\pi(\tt_0)>0$. View $\Lb$ as an infinite matrix on $\TT$, denote the $(\tt,\tt')$ element of $\Lb$ by $\Lb(\tt,\tt')$. Then for every $\tt\in\TT,\tt'\in\TT,\tt\neq\tt'$, we have
     $$\Lb(\tt,\tt)<0,\qquad\Lb(\tt,\tt')\ge0.$$
     Then, we have $$(\lab+|\Lb(\tt,\tt)|)\pi(\tt)=\sum_{\tt'\neq\tt}\Lb(\tt',\tt)\pi(\tt')\ge\Lb(\tt_0,\tt)\pi(\tt_0).$$
    Thus, we know that as long as $\pi(\tt_0)>0, \Lb(\tt_0,\tt)>0$ for some $\tt,\tt_0\in\TT$,  we can derive $\pi(\tt)>0$. Then, the conclusion is derived since we can prove that $\Lb$ is irreducible using an argument similar to that of Lemma~\ref{lem:A_nLarge}.
    \end{proof}

    The following lemma characterizes the limit of the right eigenvector.
    \begin{lemma}\label{righteigenvector}
        For every $\alpha>\max\{0,\frac{\gamma-\theta}{\gamma+\theta}\}$, there exist $c,C>0$ such that:
        $$\tilde{h}(*)=1,\qquad\frac{c}{|\tt|}\le \tilde{h}(\tt)\le C|\tt|^{\alpha},\qquad\Lb \tilde{h}= \lab\tilde{h}.$$
        When $\gamma>\theta$, we have $$\inf_{\tt\in\TT}\tilde{h}(\tt)>0.$$

    \end{lemma}
    
    \begin{proof}
        The estimates of $\tilde{h}$ are direct from Lemma~\ref{lem:eigenvectorLowerbound} and Lemma~\ref{lem:eigenvectorUpperbound}. We only need to prove $\Lb \tilde{h}=\lab \tilde{h}$. Fix $\tt$, since
        $$h_n(\tt)\xrightarrow{n\rightarrow\infty} \tilde{h}(\tt),\qquad\lab_n\xrightarrow{n\rightarrow\infty}\lab,\qquad\Lb_nh_n(\tt)=\lab_nh_n(\tt)$$
        we only need to prove $\Lb_nh_n(\tt)\rightarrow\Lb \tilde{h}(\tt)$. The convergence of the fragmentation term and the isolation term is a direct result of pointwise convergence since the terms only include subtrees of $\tt$. Regarding the growth term, we only need to prove that for any $v\in\tt$ we have
        $$\sum_{|\tt'|\le n}p_{\tt'}h_n(\tt\oplus_v\tt')\xrightarrow{n\rightarrow\infty}\sum_{\tt'}p_{\tt'}\tilde{h}(\tt\oplus_v\tt').$$
        Fix $K>|\tt|$, we use pointwise convergence to derive
        $$\sum_{|\tt'|\le K}p_{\tt'}h_n(\tt\oplus_v\tt')\xrightarrow{n\rightarrow\infty}\sum_{|\tt'|\le K}p_{\tt'}\tilde{h}(\tt\oplus_v\tt').$$
        Besides,
        \begin{multline*}
            \quad\Ll|\sum_{K<|\tt'|\le n}p_{\tt'}h_n(\tt\oplus_v\tt')-\sum_{|\tt'|>K}p_{\tt'}\tilde{h}(\tt\oplus_v\tt')\Rr|\\
            \le C\sum_{|\tt'|>K}p_{\tt'}(|\tt|+|\tt'|)^{\alpha}\le 2^{\alpha}C\sum_{|\tt'|>K}p_{\tt'}|\tt'|^{\alpha}
            \le\frac{2^{\alpha}C}{K^{2-\alpha}}\mathbf{m}_2
        \end{multline*}
        We complete the proof by first letting $n\rightarrow\infty$ and then $K\rightarrow\infty$.
    \end{proof}
    In the end, we prove Proposition \ref{eigenelements}.
    \begin{proof}[Proof of Proposition \ref{eigenelements}]
        From Lemma~\ref{lefteigenvector} and Lemma~\ref{righteigenvector} we have derived $\pi,\lab$ and $\tilde{h}$ satisfying
        $$\Lb \tilde{h}=\lab \tilde{h},\qquad\Lb^*\pi=\lab\pi,\qquad\langle\pi,|\cdot|^2\rangle<+\infty.$$
        Since $\tilde{h}(\tt)\le C|\tt|^{\alpha}$, we know that $\langle\pi,\tilde{h}\rangle<+\infty$. Thus, we set
        $$h=\frac{\tilde{h}}{\langle\pi,\tilde{h}\rangle}.$$ Then $\langle\pi,h\rangle=1$. The estimates of $h,\pi$ are direct results of Lemma~\ref{righteigenvector} and Lemma~\ref{lefteigenvector}. The only thing left to prove is the sign of $\lab$. This is direct from
        $$\lab=\langle\Lb^*\pi,1\rangle=\langle\pi,\Lb1\rangle=\langle\pi,(\gamma-\theta)(|\cdot|-1)\rangle=(\gamma-\theta)\langle\pi,|\cdot|-1\rangle$$
        and $\langle\pi,|\cdot|-1\rangle>0$ since $\pi$ is positive.
     \end{proof}

\section{Existence and uniqueness of Perron's root}

In Proposition~\ref{eigenelements} we have shown the existence of eigenelements of $\Lb$. In this section, we consider the original model instead of the modified model. Since $\L=\Lb-\theta I$, we have
$$\L h=\lambda h,\qquad\L^*\pi=\lambda\pi,\qquad\lambda=\lab-\theta .$$
Formally, we can define the first moment semigroup $(M_s)_{s\ge0} $ by 
$M_s=\exp(s\L), $ which satisfies
$$M_sh=e^{\lambda s}h,\qquad\pi M_s=e^{\lambda s}\pi.$$
This section aims to define the semigroup rigorously and prove the Perron-Frobenius Theorem.

\subsection{Well-definedness of the first moment semigroup}
In this subsection, we define $M_t$ as a bounded semigroup on $\B(|\cdot|^2)$ .  For a nonnegative function $f$, we define 
$$M_sf(\tt):=\mathbb{E}_{\delta_{\tt}}\Ll[\langle X_s, f\rangle\Rr], \qquad s \ge 0,\; \tt \in \TT. $$
Here, $\mathbb{P}_{\delta_{\tt}}$ denotes the law of the process starting from $X_0=\delta_{\tt}$ and $\mathbb{E}_{\delta_{\tt}}$ is the expectation associated with it. Then we extend the definition from the positive cone of $\B(|\cdot|^2)$ to the entire space, and show that $\L$ actually generates the semigroup $(M_s)_{s\geq 0}$. 

\begin{lemma}\label{lem.well-definedsemigroup}
    \begin{enumerate}
        \item For every $1\leq q\leq 2, s\geq 0, \tt\in\TT$, we have
        $$M_s(|\cdot|^q)(\tt)\leq e^{2^{q-1}\beta q \mathbf{m}_{q}s}|\tt|^q. $$
        \item For all
        $s\ge 0,\ $ $M_s$ is a well-defined bounded operator from $\B(|\cdot|^2)$ to $\B(|\cdot|^2)$.
        \item $(M_s)_{s\geq 0}$ is a weakly continuous positive semigroup on $\mathcal{B}(|\cdot|^2)$, and for every $f\in\mathcal{B}(|\cdot|)$, we have
            \begin{align}
            \label{eq:LgeneratesM_t}
                \frac{d}{ds}M_sf(\tt)=M_s(\L f)(\tt). 
            \end{align}
    \end{enumerate}
\end{lemma}

\begin{proof}
    We first prove $(1).$ We first study $\L(|\cdot|^q)$. Notice that for every $a,b>0,q\ge1$, we have $a^q+b^q\le(a+b)^q$. Thus, the fragmentation and isolation term in $\L(|\cdot|^q)$ is non-positive. Then we have 
    \begin{align*}\label{eq.linearoperator}
        \L (|\cdot|^q)(\tt)
        \le\beta\sum_{v\in\tt}\sum_{\tt'}p_{\tt'}\Ll((|\tt|+|\tt'|)^q-|\tt|^q\Rr).   
    \end{align*}
    We apply convexity for the function $|\cdot|^q$ and use $|\tt|+|\tt'|\le2|\tt||\tt'|$:
    $$(|\tt|+|\tt'|)^q-|\tt|^q\leq q|\tt'|(|\tt|+|\tt'|)^{q-1}\le 2^{q-1}q|\tt'|^q|\tt|^{q-1}.$$
    Thus, we have
    $$\L(|\cdot|^q)(\tt)\leq \beta\sum_{v\in\tt}\sum_{\tt'}p_{\tt'}q2^{q-1}|\tt|^{q-1}|\tt'|^{q}
        =2^{q-1}\beta q\mathbf{m}_q|\tt|^q. $$
    For any fixed $\tt$ and $X_0=\delta_{\tt}$, we consider the stopping time
    $$\tau_m:=\inf\Ll\{s\in \mathbb{R}_{+}: \langle X_s,|\cdot|\rangle\geq m\Rr\}$$
    and the associated stopped process $(X_s^m)_{s\geq0}$ defined by $X_s^m=X_{s\wedge \tau_m}$. Applying Dynkin's formula to $|\cdot|^q$ and the stopped process yields
    $$\mathbb{E}_{\delta_{\tt}}[\langle X_s^m, |\cdot|^q\rangle]=\mathbb{E}_{\delta_{\tt}}[\langle X_0^m,|\cdot|^q\rangle]+\mathbb{E}_{\delta_{\tt}}\Ll[\int_0^{s\wedge \tau_m}\langle X_r,\L |\cdot|^q\rangle \,dr\Rr]. $$
    Furthermore, from the result above we have 
    $$\mathbb{E}_{\delta_{\tt}}[\langle X_s^m, |\cdot|^q\rangle]\leq \mathbb{E}_{\delta_{\tt}}[\langle X_0^m, |\cdot|^q\rangle]+2^{q-1}\beta q\mathbf{m}_{q}\mathbb{E}_{\delta_{\tt}}\Ll[\int_0^{s\wedge \tau_m}\langle X_r, |\cdot|^q\rangle\,dr\Rr]. $$
   Letting $m\rightarrow\infty$, applying Fatou's lemma on the left and the monotone convergence theorem on the right, we obtain
    $$M_s(|\cdot|^q)(\tt)\leq M_0(|\cdot|^q)(\tt)+2^{q-1}\beta q\mathbf{m}_{q}\int_0^s M_r(|\cdot|^q)(\tt)\,dr. $$
    Gr\"{o}nwall's lemma then yields $(1)$.

    Now we prove $(2)$. For any $f\in \mathcal{B}(|\cdot|^2)$, we set $f_+$ and $f_-$ to be the positive part and the negative part of f respectively. Since $M_s$ is positive, using the estimate in $(1)$ we have
    \begin{align*}
    &M_s f_{+}\le M_s(\norm{f_+}_{\B(|\cdot|^2)}|\cdot^2|)\le C_s\norm{f_+}_{\B(|\cdot|^2)}|\cdot|^2,\\
    &M_s f_{-}\le M_s(\norm{f_-}_{\B(|\cdot|^2)}|\cdot^2|)\le C_s\norm{f_-}_{\B(|\cdot|^2)}|\cdot|^2.
    \end{align*}
    Thus, we have 
    $$\norm{M_sf_+}_{\B(|\cdot|^2)}\le C_s\norm{f_+}_{\B(|\cdot|^2)},\qquad\norm{M_sf_-}_{\B(|\cdot|^2)}\le C_s\norm{f_-}_{\B(|\cdot|^2)}.$$
    Then we define
    \begin{align*}
        M_sf:=M_sf_+-M_sf_-\in\B(|\cdot|^2).
    \end{align*}
    We have
    \begin{align*}   \norm{M_sf}_{\B(|\cdot|^2)}\le\norm{M_sf_+}_{\B(|\cdot|^2)}+\norm{M_sf_-}_{\B(|\cdot|^2)}\le C_s(\norm{f_+}_{\B(|\cdot|^2)}+\norm{f_-}_{\B(|\cdot|^2)})\le2C_s\norm{f}_{\B(|\cdot|^2)}
    \end{align*}
    Thus $M_s$ is a bounded operator.
    
    Now we prove $(3)$. We first prove that $(M_s)_{s\ge0}$ is weakly continuous on $\B(|\cdot|^2)$. For $f\in\B(|\cdot|^2)$, without loss of generality, we assume
    that $|f(\tt)|\le|\tt|^2,\forall\tt\in\TT$. It suffices to prove that fix $\tt\in\TT$, the function
    $$s \mapsto M_sf(\tt)=\E_{\delta_{\tt}}\Ll[\langle X_s,f\rangle\Rr]$$
    is continuous. Assume $\lim_{n\rightarrow\infty}s_n=s_0,$ we prove that 
    \begin{equation}
    \label{eq:M_tcontinuous}
    \lim_{n\rightarrow\infty}\E_{\delta_{\tt}}\Ll[\langle X_{s_n},f\rangle\Rr]=\E_{\delta_{\tt}}\Ll[\langle X_{s_0},f\rangle\Rr].
    \end{equation}
    We consider the Markov process $(Y_s)_{s\ge 0}$ with only growth and no fragmentation or isolation. It is a Markov process on $\TT$ with generator
    $$\tilde{\L}\phi(\tt)=\beta\sum_{v\in\tt}\sum_{\tt'\in\TT}p_{\tt'}(\phi(\tt\oplus_v\tt')-\phi(\tt)).$$
    Using $\tilde{\L}|\cdot|^2\le C|\cdot|^2$ and a localization argument just like what we did in the proof of $(1)$, we can prove that $\E_{\delta_{\tt}} [|Y_s|^2]<+\infty,\forall s\ge 0$. Since $(Y_s)_{s\ge 0}$ is the process with only growth but no fragmentation or isolation, we can couple $(X_s)_{s\ge0}$ and $(Y_s)_{s\ge 0}$ on the same probability space such that almost surely for any $s\ge0, \langle X_s,|\cdot|\rangle\le |Y_s|$.  Then $\langle X_{s_n},|\cdot|^2\rangle\le|Y_{s_n}|^2\le|Y_T|^2$ for some $T>\sup_{n}s_n$. Then 
    $\E_{\delta_{\tt}} [|Y_T|^2]<+\infty$ and the dominated convergence theorem give us \eqref{eq:M_tcontinuous}.
    
    Now, we prove that for any $f\in\B(|\cdot|)$, we have
    $$\frac{d}{ds}M_sf(\tt)=M_s(\L f)(\tt).$$
    Since $f\in\B(|\cdot|)$, we can take $C>0$ such that $|f(\tt)|\leq C |\tt|$ and $|\L f(\tt)|\leq C|\tt|^2$ for all $\tt$. Thus, the function
    $$s \mapsto \E_{\delta_{\tt}}[\langle X_s,\L f\rangle]$$
    is continuous. Then it suffices to prove
    \begin{equation}\label{eq.targetequation}
        \mathbb{E}_{\delta_{\tt}}[\langle X_s,f\rangle]=\mathbb{E}_{\delta_{\tt}}[\langle X_0,f\rangle]+\int_0^s\mathbb{E}_{\delta_{\tt}}[\langle X_r, \L f\rangle]dr. 
    \end{equation}
   Dynkin's formula gives
    \begin{equation}\label{eq.Dynkin}
    \mathbb{E}_{\delta_{\tt}}[\langle X_s^m, f\rangle]=\mathbb{E}_{\delta_{\tt}}[\langle X_0^m,f\rangle]+\mathbb{E}_{\delta_{\tt}}\Ll[\int_0^{s\wedge \tau_m}\langle X_r,\L f\rangle\,dr\Rr]. 
    \end{equation}
    We will show that as $m\rightarrow\infty$ both sides of \eqref{eq.Dynkin} converge to the corresponding sides of \eqref{eq.targetequation}. 

    The convergence of the left-hand side term is a direct result of the dominated convergence theorem, since $\langle X_{s}^m,f\rangle\le C|Y_s|\in L^1$ and $\lim_{m\rightarrow\infty}\langle X_s^m,f\rangle=\langle X_s,f\rangle$ a.s. Now we turn to the right-hand side. Almost surely, we have 
    $$\int_0^{s \wedge \tau_m} \langle X_r, \L f \rangle \, dr\xrightarrow{m\rightarrow\infty}\int_0^s \langle X_r, \L f \rangle \, dr.$$ 
    Note that
    $$\big|\int_0^{s \wedge \tau_m} \langle X_r, \L f \rangle\, dr \big| \leq \int_0^s \langle X_r, |\L f| \rangle\, dr \leq C \int_0^s \langle X_r, |\cdot|^2 \rangle \, dr. $$

    Due to $(1)$, the last term has a finite mean. Thus, we can apply the dominated convergence theorem and Fubini's theorem to obtain 
    $$\mathbb{E}_{\delta_{\tt}} \left[ \int_0^{s \wedge \tau_m} \langle X_r, \L f \rangle \, dr \right] \xrightarrow{m \to \infty} \mathbb{E}_{\delta_{\tt}} \left[ \int_0^s \langle X_r, \L f \rangle \, dr \right]=\int_0^s \mathbb{E}_{\delta_{\tt}}[\langle X_r, \L f \rangle] \, dr$$

    Thus letting $m \to \infty$ in \eqref{eq.Dynkin} gives \eqref{eq.targetequation} and completes the proof. 
    \end{proof}

\subsection{Perron–Frobenius Theorem }

    In this subsection, we will prove Theorem~\ref{thm.main}. 

    We intend to apply the results from \cite{bansaye2019non} to derive the conclusions. However, since only the second moment is assumed,  we cannot derive that $\L$ generates the semigroup $(M_t)_{t\geq 0}$ in the space $\B(|\cdot|^2)$. Fortunately, this issue can be circumvented by directly verifying \cite[Assumption~(A1), (A2)]{bansaye2019non} together with  \cite[Proposition~2.3]{bansaye2019non}, which is summarized as the following lemma.
    
    %instead of applying Proposition~2.2 like previous work.

    \begin{lemma}\label{lem.lyapunovcondition}
        Let $V=|\, \cdot \,|^2$ and $h$ be the right eigenvector of $\L$, then there exist constants ${\tau>0}, \ {0<a<e^{\lambda \tau}},\ {b\geq 0}$ and $\ K\subset \TT$ such that: 
        \begin{enumerate}
            \item $M_{\tau}V\leq aV+b\Ind{K}h$, where $K$ is a non-empty finite subset of $\TT$;
            \item $M_{\tau}h=e^{\lambda\tau}h$;
            \item For any $\tt_1, \tt_2\in K$, 
            $$\delta_{\tt_1}M_{\tau}(\Ind{\tt_2})>0. $$
        \end{enumerate}
    \end{lemma}

    To prove the above lemma, we shall make use of the following variant of Gr\"{o}nwall's lemma.

    \begin{lemma}\label{lem.Gronwallvariation}
        Let $g: [0, \infty)\to\mathbb{R}$ be a continuous function, $C \in \R$,  and $c \in \R_+$. Assume the following estimate for all $0\leq s\leq t$
        \begin{equation}\label{eq.Gronwallvarationcondition}
        g(t) - g(s) \leq \int_s^t(-cg(r)+C)\,dr.
        \end{equation}
        Then for any $t\geq 0$ 
        \begin{equation}\label{eq.Gronwallvariation}
        g(t) \leq \frac{C}{c} + \left[ g(0) - \frac{C}{c} \right]_+ e^{-ct},
        \end{equation}
        where $[ \cdot ]_+ := \max(0, \cdot)$.
    \end{lemma}
    The proof of the lemma is postponed to the appendix. Now we prove Lemma~\ref{lem.lyapunovcondition}.
    \begin{proof}[Proof of Lemma~\ref{lem.lyapunovcondition}]
        (2) is a corollary of Lemma~\ref{lem.well-definedsemigroup}$(3)$ and $h\in\B(|\cdot|)$ from Proposition~\ref{eigenelements}. Thanks to the irreducibility of the process, (3) holds automatically as long as the K in (1) exists. Now we focus on the proof of (1). 
        
        We choose $\tau=1$. Similar to the proof of Lemma~\ref{lem.well-definedsemigroup}, we consider the stopping times $\{\tau_m\}_m$ and the stopped process $\{X_s^m\}_{s\ge 0}$. We apply Dynkin's formula to the stopped process to obtain that for any $0\leq s\leq t\le 1$, 
        $$\mathbb{E}_{\delta_{\tt}}[\langle X_t^m, V\rangle]=\mathbb{E}_{\delta_{\tt}}[\langle X_s^m,V\rangle]+\mathbb{E}_{\delta_{\tt}}\Ll[\int_{s\wedge\tau_m}^{t\wedge \tau_m}\langle X_r,\L V\rangle dr\Rr]. $$
        From
        \begin{align*}
            \L V(\tt)&=\beta\sum_{v\in\tt}\sum_{\tt'\in\TT}p_{\tt'}((|\tt|+|\tt'|)^2-|\tt|^2)-\theta|\tt|^3+\gamma\sum_{v\in\tt\setminus\{\rt\}}(|\tt_v|^2+|\tt\ominus\tt_v|^2-|\tt|^2)\\
            &\leq -\theta|\tt|^3+2\beta\mathbf{m}_1|\tt|^2+\beta\mathbf{m}_2|\tt|
        \end{align*}
        and $h>0$ we know that there exist constants $c>0,\ C>0$ such that $\lambda-c<0$ and for any $\tt\in \TT$ 
        $$\L V(\tt)\leq (\lambda-c)V(\tt)+Ch(\tt). $$
        Then 
        \begin{align*}
            \mathbb{E}_{\delta_{\tt}}[\langle X_t^m, V\rangle]-\mathbb{E}_{\delta_{\tt}}[\langle X_s^m,V\rangle]\leq \mathbb{E}_{\delta_{\tt}}\Ll[\int_{s\wedge\tau_m}^{t\wedge \tau_m}\langle X_r,(\lambda-c)V+Ch\rangle\,dr\Rr].
        \end{align*}
        Letting $m\rightarrow\infty$, the same dominated convergence theorem argument as in (3) of Lemma \ref{lem.well-definedsemigroup} shows that
        $$\mathbb{E}_{\delta_{\tt}}[\langle X_t, V\rangle]-\mathbb{E}_{\delta_{\tt}}[\langle X_s,V\rangle]\leq \mathbb{E}_{\delta_{\tt}}\Ll[\int_{s}^{t}\langle X_r,(\lambda-c)V+Ch\rangle\,dr\Rr].$$
        Thus it results in
        \begin{align*}
        M_tV(\tt)-M_sV(\tt)&\leq \int_s^t \Ll((\lambda-c)M_rV(\tt)+Ce^{\lambda r}h(\tt)\Rr)\,dr \\
        &\le\int_{s}^t\Ll(-|\lambda-c|M_rV(\tt)+Ce^{|\lambda|}h(\tt)\Rr)\,dr
        \end{align*}
        Since Lemma~\ref{lem.well-definedsemigroup}$(3)$ gives us the continuity of the function $s\rightarrow M_sV(\tt)$, applying Lemma~\ref{lem.Gronwallvariation} to it yields
        $$M_1V(\tt)\leq \frac{Ce^{|\lambda|}h(\tt)}{|\lambda-c|}+V(\tt)e^{\lambda-c}=e^{\lambda-\frac{c}{2}}V(\tt)+\Ll(C'h(\tt)-(e^{\lambda-\frac{c}{2}}-e^{\lambda-c})V(\tt)\Rr). $$
        We complete the proof by choosing $a=e^{\lambda-\frac{c}{2}}<e^{\lambda}$ and $b,K$ large enough such that
        $$C'h(\tt)-(e^{\lambda-\frac{c}{2}}-e^{\lambda-c})V(\tt)\le b\Ind{K}h(\tt).$$
    \end{proof}
    
    \begin{proof}[Proof of Theorem \ref{thm.main}]
        Lemma~\ref{lem.lyapunovcondition} and Lemma~\ref{lem.well-definedsemigroup} ensure that the semigroup $(M_t)_{t\ge0}$  is a local bounded positive semigroup on $\B(|\cdot|^2)$ and satisfies Assumptions (A1) and (A2) of \cite{bansaye2019non} for $V=|\cdot|^2$.  Moreover, Proposition 2.3 of the work implies that Assumptions (A3) and (A4) are satisfied. Thus applying Theorem 2.1 in \cite{bansaye2019non} gives the existence and uniqueness(up to scaling of $\pi'$ and $h'$) of the triplet $(\lambda',\pi',h')$ satisfying
        $$\pi',h'>0,\qquad M_sh'=e^{\lambda's}h',\qquad \pi'M_s=e^{\lambda's}\pi'.$$
        and 
        \begin{align*}
            \sup_{\substack{f:\TT \to \mathbb{R} \\
        \sup_{\tt\in \TT} \vert f(\tt) \vert / \vert \tt \vert^2  \leq 1}}\big\vert e^{-\lambda' s} M_s f(\tt) -  h'(\tt) \langle\pi',f\rangle \big\vert  &\leq C \vert \tt \vert^{2}  e^{-\omega s}.
        \end{align*}
        On the other hand, Lemma~\ref{lem.well-definedsemigroup}$(3)$ and $h\in\B(|\cdot|)$ derives $M_sh=e^{\lambda s}h$. For any $f\in\B(|\cdot|)$, we have 
        $$\frac{d}{ds}\langle\pi, M_sf\rangle=\langle\pi,\L M_s f\rangle=\langle\L^*\pi,M_sf\rangle=\lambda\langle\pi,M_sf\rangle.$$
        Thus, we have $\langle\pi,M_sf\rangle=e^{\lambda s}\langle\pi,f\rangle$ for all $f\in\B(|\cdot|)$. So $\pi M_s=e^{\lambda s}\pi$. Combining with the uniqueness of the triplet, we know that after rescaling of $h'$ and $\pi'$, we have
        $$\lambda'=\lambda,\qquad\pi'=\pi,\qquad h'=h.$$
        Thus we complete the proof.
    \end{proof}

\section{$L^1$ convergence of the process}

 In the previous sections, the Perron's root $\lambda\in\mathbb{R}$ and associated eigenelements are derived. The sign of $\lambda$ determines whether the first moment semigroup goes to $0$ or to infinity. In this section, we discuss the trajectorial results of the process.
 The goal is to prove Theorem~\ref{Thm:a.e.ConvergenceIntro}. We use \cite[Theorem~3.1]{bansaye2025stronglawlargenumbers} as the main tool, and the crucial condition to be verified is the following proposition:
 \begin{proposition}
 \label{Pro:ConditionOfa.e.ConvergenceTheorem}
    Assume $\lambda>0$. There exist two positive functions $V$ and $V^*$ on $\TT$ such that $h\in\mathcal{B}(V)\subset \mathcal{B}(V^*)$ and $\langle\pi,V^*\rangle<+\infty$. Moreover, there exist $C>0$ and $\omega>0$ such that for all $s\ge 0$, we have
    \begin{equation}
        \label{Eq:V*lyapunov}
        \sup_{|f|\leq V^*}|e^{-\lambda s}M_s f(\tt)-h(\tt)\langle\pi,f\rangle|\leq CV^*(\tt)e^{-\omega s},\forall\tt\in\TT.
    \end{equation}
    For any $s\in(0,1]$ we have
    \begin{equation}
    \label{Eq:ExpectationOfZ_tFinite}
        \sup_{\tt\in\TT}\frac{\mathbb{E}_{\delta_{\tt}}\Ll[\langle X_s,V\rangle\log^*\langle X_s,V\rangle\Rr]}{V^*(\tt)}<+\infty,
    \end{equation}
    where
    \[
    \log^*{x}=
    \begin{cases}
    \frac{x}{e},&0\leq x\leq e\\
    \log(x),&x>e
    \end{cases}
    \]
 \end{proposition}
\begin{proof}
    Set $V(\tt)=h(\tt)$ and $V^*(\tt)=|\tt|^2$. 
    By \eqref{eq:EstimationOfH}, $h(\tt)\leq C|\tt|^\alpha$ for some $0<\alpha<1$, so $h\in\mathcal{B}(h)\subset\mathcal{B}(V^*)$. 
    Moreover, \eqref{eq:MomentOfPi} gives $\langle\pi,V^*\rangle<+\infty$, and Theorem \ref{thm.main} gives \eqref{Eq:V*lyapunov}. 
    It remains only to verify \eqref{Eq:ExpectationOfZ_tFinite}. 
    
    Recall that \(\mathcal{M}_F(\TT)\) is the set of all finite point measures on \(\mathcal{T}\), which is precisely the space of all finite forests. We denote a forest in $\mathcal{M}_F(\TT)$ by $\mathbf{F}$. Then the GFI process is a Markov jump process on $\mathcal{M}_F(\TT)$ with generator    
    \begin{align*}    \mathscr{L}\Phi(\mathbf{F})&=\sum_{\mathbf{F}'\in\mathcal{M}_F(\TT)}q(F,F')\left(\Phi(\mathbf{F'})-\Phi(\mathbf{F})\right)=\mathscr{L}_G\Phi(\mathbf{F})+\mathscr{L}_F\Phi(\mathbf{F})+\mathscr{L}_I\Phi(\mathbf{F}),\\ \end{align*}
    where
    \[
    \begin{cases}
    \mathscr{L}_G\Phi(\mathbf{F})=\beta\sum_{\tt\in\mathbf{F}}\sum_{v\in\tt}\sum_{\tt'\in\TT}p_{\tt'}\left(\Phi(\mathbf{F}-\tt+\tt\oplus_{v}\tt')-\Phi(\mathbf{F})\right)\\
    \mathscr{L}_F\Phi(\mathbf{F})=\gamma\sum_{\tt\in\mathbf{F}}\sum_{v\in\tt\setminus{\{\rt\}}}(\Phi(\mathbf{F}-\tt+\tt_v+\tt\ominus\tt_v)-\Phi(\mathbf{F}))\\
    
    \mathscr{L}_I\Phi(\mathbf{F})=\theta\sum_{\tt\in\mathbf{F}}|\tt|(\Phi(\mathbf{F}-\tt)-\Phi(\mathbf{F}))
    \end{cases}
    \]

    To prove \eqref{Eq:ExpectationOfZ_tFinite}, we take $$\Phi(\mathbf{F})=\langle\mathbf{F},h\rangle\log^*\langle\mathbf{F},h\rangle$$
    and try to apply  Dynkin's formula to $\Phi$. Let 
    $$\tau_m=\inf_{s\geq0}\Ll\{\langle X_s,|\cdot|\rangle\geq m\Rr\},\qquad X_s^m=X_{s\wedge\tau_m}.$$
    We now prove that
    \begin{equation}                                            \label{Eq:DynkinForTruncated}\mathbb{E}_{\delta_{\tt}}\Ll[\Phi(X_s^m)\Rr]=\Phi(\delta_{\tt})+\mathbb{E}_{\delta_{\tt}}\Ll[\int_{0}^{s\wedge\tau_m}\mathscr{L}\Phi(X_r)dr\Rr].
    \end{equation}
    By Dynkin's formula, we only need to verify that 
    \begin{equation}
    \label{Eq:ConditionOfDynkin}
    \mathbb{E}_{\delta_{\tt}}\Ll[\int_{0}^{s\wedge\tau_m}\sum_{\mathbf{F'}\in\mathcal{M}_F(\TT)}q(X_r,\mathbf{F}')|\Phi(\mathbf{F}')-\Phi(X_r)|dr\Rr]<+\infty.
    \end{equation}
    We denote by $Z_s:=\langle X_s,h\rangle$ the mass of $X_s$, and when the event $X_s$ turns to $\mathbf{F'}$ occurs, we denote by $a_{\mathbf{F'}}$ the mass that is removed and by $R_\mathbf{F'}$ the mass that is added. To be more specific, when an event happens to the cluster $\tt$ of $X_s$, we define:
    \[
    a_{\mathbf{F'}}=h(\tt)\qquad R_{\mathbf{F'}}=
    \begin{cases}
    h(\tt\oplus_v\tt'),&\text{when $\tt'$ is attached to $v$}\\
    h(\tt_v)+h(\tt\ominus\tt_v),&\text{when the edge above $v$ is fragmented}\\
    0,&\text{when isolation happens}
    \end{cases}
    \]
    Note that we always have $a_\mathbf{F'}\le Z_s$. It can be verified that 
    \begin{align}
        \label{Eq:CalOfMinusPhi}
        \Phi(\mathbf{F'})-\Phi(X_s)&=\underbrace{(Z_s-a_{\mathbf{F'}}+R_{\mathbf{F'}})\Ll(\log^*(Z_s-a_{\mathbf{F'}}+R_{\mathbf{F'}})-\log^*(Z_s)\Rr)}_{A_{\mathbf{F'}}}\notag\\
        &+\underbrace{(R_{\mathbf{F'}}-a_{\mathbf{F'}})\log^*(Z_s)}_{B_{\mathbf{F'}}}.
    \end{align}
    To estimate the absolute value of $A_{\mathbf{F'}}$, we establish the following inequality:
    \begin{equation}
    \label{Eq:EstimateOfEnt}
        |(Z_s-a_{\mathbf{F'}}+R_{\mathbf{F'}})\left(\log^*(Z_s-a_{\mathbf{F'}}+R_{\mathbf{F'}})-\log^*(Z_s)\right)|\leq(e+a_{\mathbf{F'}}\vee R_{\mathbf{F'}})\log^*(e+CR_{\mathbf{F'}}|\tt|)
    \end{equation}
    for some constant $C>0$. We first prove the upper bound.
    
    When $R_{\mathbf{F'}}\leq a_{\mathbf{F'}}$, 
    $$LHS:=(Z_s-a_{\mathbf{F'}}+R_{\mathbf{F'}})\left(\log^*(Z_s-a_{\mathbf{F'}}+R_{\mathbf{F'}})-\log^*(Z_s)\right)\leq0$$
    When $R_{\mathbf{F'}}>a_{\mathbf{F'}}$ and $Z_s\geq e$, using that the function $$x\rightarrow (x+R_{\mathbf{F'}}-a_{\mathbf{F'}})\log\Ll(\frac{x+R_{\mathbf{F'}}-a_{\mathbf{F'}}}{x}\Rr)$$ is decreasing and $Z_s\geq a_{\mathbf{F'}},a_{\mathbf{F'}}=h(\tt)\geq \frac{c}{|\tt|}$ from \eqref{eq:EstimationOfH}, we have 
    $$LHS\leq R_{\mathbf{F'}}\log\Ll(\frac{R_{\mathbf{F'}}}{a_{\mathbf{F'}}}\Rr)
        \le (e+R_{\mathbf{F'}})\log^*(e+c R_{\mathbf{F'}}|\tt|)$$
    And finally when $R_{\mathbf{F'}}>a_{\mathbf{F'}}$ and $Z_s\leq e$, we have
    $$LHS\leq(e+R_{\mathbf{F'}})\log^*(e+R_{\mathbf{F'}})$$
    The above inequalities give the upper bound. The lower bound can be obtained in  the same way. Now we prove \eqref{Eq:ConditionOfDynkin}: since $r\le\tau_m$, we have $\langle X_r,|\tt|\rangle\leq m$. \eqref{eq:EstimationOfH} gives us $h(\tt)\le C|\tt|^{\alpha}$ for some $0<\alpha<1$, so 
    \begin{equation}
    \label{Eq:Z_rBounded}
    Z_r=\langle X_r,h\rangle\leq Cm
    \end{equation}
    Let $G(X_r),F(X_r),I(X_r)$ denote respectively the set of forests obtained from $X_r$ by growth, fragmentation, and isolation. We ignore the change of constants in the following proof. Then we have
    \begin{align*}
        \sum_{\mathbf{F'}\in\mathcal{M}_F(\TT)}q(X_r,\mathbf{F}')|\Phi(\mathbf{F}')-\Phi(X_r)|
        &=\sum_{\mathbf{F'}\in G(X_r)}q(X_r,\mathbf{F}')|\Phi(\mathbf{F}')-\Phi(X_r)|\\
        &\qquad +\sum_{\mathbf{F'}\in F(X_r)}q(X_r,\mathbf{F}')|\Phi(\mathbf{F}')-\Phi(X_r)|\\
        &\qquad +\sum_{\mathbf{F'}\in I(X_r)}q(X_r,\mathbf{F}')|\Phi(\mathbf{F}')-\Phi(X_r)|.
    \end{align*}
    For $\mathbf{F'}$ obtained by growth, using \eqref{Eq:CalOfMinusPhi} we have
    $$\sum_{\mathbf{F'}\in G(X_r)}q(X_r,\mathbf{F}')|\Phi(\mathbf{F}')-\Phi(X_r)|
    \le\sum_{\mathbf{F'}\in G(X_r)}q(X_r,\mathbf{F}')(|A_{\mathbf{F'}}|+|B_{\mathbf{F'}}|).$$
    Assume $\mathbf{F}'$ is obtained by attaching $\tt'$ on $v$ of $\tt$. Then 
    $$a_{\mathbf{F}'}\le C|\tt|^{\alpha},\qquad R_{\mathbf{F}'}\le C(|\tt|+|\tt'|)^{\alpha}.$$
    Using \eqref{Eq:EstimateOfEnt} and \eqref{Eq:Z_rBounded} we have
    \begin{align*}
    &|A_{\mathbf{F}'}|\le C(|\tt|+|\tt'|)^{\alpha}\log^{*}(C(|\tt|+|\tt'|)^{1+\alpha})\le C(|\tt|+|\tt'|),\\
    &|B_{\mathbf{F}'}|\le C(|\tt|+|\tt'|)^{\alpha}\log^*(m)\le C\log^*m(|\tt|+|\tt'|).
    \end{align*}
    Thus, we have
    \begin{align}
    \sum_{\mathbf{F}'\in G(X_r)}q(X_r,\mathbf{F}')|A_{\mathbf{F}'}|&\le\beta\sum_{\tt\in X_r}\sum_{v\in\tt}\sum_{\tt'\in\TT}p_{\tt'}C(|\tt|+|\tt'|)\notag\\
    &=C\beta\Ll(\langle X_r,|\cdot|^2\rangle+\langle X_r,|\cdot|\rangle\mathbf{m}_1\Rr)\notag\\
    &\le C\langle X_r,|\cdot|^2\rangle\label{Eq:EstimateOfA_F}
    \end{align}
    for some constant $C$ independent of $m$. The same computation gives
    $$\sum_{\mathbf{F}'\in G(X_r)}q(X_r,\mathbf{F}')|B_{\mathbf{F}'}|\le C\log^*m \langle X_r,|\cdot|^2\rangle.$$
    Thus, we have
    $$\sum_{\mathbf{F'}\in G(X_r)}q(X_r,\mathbf{F}')|\Phi(\mathbf{F}')-\Phi(X_r)|\le C\log^*m\langle X_r,|\cdot|^2\rangle$$
    For $\mathbf{F'}$ obtained by fragmentation or isolation on $\tt$, we have 
    $$a_{\mathbf{F}'}\vee R_{\mathbf{F}'}\le C|\tt|^{\alpha}.$$
    Again, we use \eqref{Eq:EstimateOfEnt} to obtain
    \begin{equation}
    \sum_{\mathbf{F}'\in F(X_r)\cup I(X_r)}q(X_r,\mathbf{F}')|A_{\mathbf{F}'}|\le(\gamma+\theta)\sum_{\tt\in X_r}|\tt|\cdot C|\tt|^{\alpha}\log^*|\tt|\le C\langle X_r,|\cdot|^2\rangle
    \label{Eq:EstimateOfA_F,IF}
    \end{equation}
   for some constant $C>0$ independent of $m$. The same computation gives
   $$\sum_{\mathbf{F}'\in F(X_r)\cup I(X_r)}q(X_r,\mathbf{F}')|B_{\mathbf{F}'}|\le C\log^*m \langle X_r,|\cdot|^2\rangle.$$
    Thus, we have
    $$\sum_{\mathbf{F}'\in F(X_r)\cup I(X_r)}q(X_r,\mathbf{F}')|\Phi(\mathbf{F}')-\Phi(X_r)|\le C_m\langle X_r,|\cdot|^2\rangle.$$
    Combining the results above gives 
    $$\sum_{\mathbf{F'}\in\mathcal{M}_F(\TT)}q(X_r,\mathbf{F}')|\Phi(\mathbf{F}')-\Phi(X_r)|\leq C_m\langle X_r,|\tt|^2\rangle.$$
    Thus, we have
    \begin{align*}   &\quad\mathbb{E}_{\delta_{\tt}}\Ll[\int_{0}^{s\wedge\tau_m}\sum_{\mathbf{F'}\in\mathcal{M}_F(\TT)}q(X_r,\mathbf{F}')|\Phi(\mathbf{F}')-\Phi(X_r)|dr\Rr]\\
    &\leq\mathbb{E}_{\delta_{\tt}}\Ll[\int_{0}^{s\wedge\tau_m}C_m\langle X_r,|\cdot|^2\rangle dr\Rr]\leq\int_{0}^{s}C_m\mathbb{E}_{\delta_{\tt}}\Ll[\langle X_r,|\cdot|^2\rangle\Rr] dr\\
    &=C_m\int_{0}^{s}M_r|\cdot|^2(\tt) dr\leq C_mC_s|\tt|^2<+\infty
    \end{align*}
    The inequality in the last line holds since from Lemma~\ref{lem.well-definedsemigroup}$(1)$
    we know that $M_r|\cdot|^2\leq C_r|\cdot|^2$. And finally we have proved \eqref{Eq:ConditionOfDynkin} and get Dynkin's formula \eqref{Eq:DynkinForTruncated}.
    Now we use \eqref{Eq:CalOfMinusPhi}:
    \begin{align*}
    \mathbb{E}_{\delta_{\tt}}[\mathscr{L}\Phi(X_r)]&=\mathbb{E}_{\delta_{\tt}}\Ll[\sum_{\mathbf{F'}\in\mathcal{M}_F(\TT)}q(X_r,\mathbf{F'})A_{\mathbf{F}'}\Rr]
    +\mathbb{E}_{\delta_{\tt}}\Ll[\log^*(Z_r)\sum_{\mathbf{F'}\in\mathcal{M}_F(\TT)}q(X_r,\mathbf{F'})(R_{\mathbf{F'}}-a_{\mathbf{F'}})\Rr].
    \end{align*}
    Using \eqref{Eq:EstimateOfA_F} and \eqref{Eq:EstimateOfA_F,IF}, we have:
   $$\sum_{\mathbf{F'}\in\mathcal{M}_F(\TT)}q(X_r,\mathbf{F'})A_{\mathbf{F}'}\le C\langle X_r,|\cdot|^2\rangle.$$
    Since the function $\mathbf{F'}\rightarrow\langle \mathbf{F'},h\rangle$ is a linear statistic, using the definition of the first moment semigroup and $\mathcal{L}h=\lambda h$:
    $$\sum_{\mathbf{F'}\in\mathcal{M}_F(\TT)}q(X_r,\mathbf{F'})(\langle \mathbf{F'},h\rangle-\langle X_r,h\rangle)=\sum_{\tt\in X_r}\mathcal{L}h(\tt)=\lambda \langle X_r,h\rangle=\lambda Z_r$$
    So we have
    \begin{equation}\label{Eq:EstimateOfLPhi}\mathbb{E}_{\delta_{\tt}}[\mathscr{L}\Phi(X_r)]\leq \lambda \mathbb{E}_{\delta_{\tt}}[Z_r\log^*(Z_r)]+C\mathbb{E}_{\delta_{\tt}}[\langle X_r,|\cdot|^2\rangle]\leq \lambda \mathbb{E}_{\delta_{\tt}}[\Phi(X_r)]+C|\tt|^2\end{equation}
     The last inequality is obtained by using Lemma~\ref{lem.well-definedsemigroup}$(1)$ and $r\le 1$:
    $$\mathbb{E}_{\delta_{\tt}}[\langle X_r,|\cdot|^2\rangle]=M_rf_2(\tt)\le C_r|\tt|^2\leq C|\tt|^2.$$
    Inserting \eqref{Eq:EstimateOfLPhi} into Dynkin's formula \eqref{Eq:DynkinForTruncated}:
    $$\mathbb{E}_{\delta_{\tt}}[\Phi(X_s^m)]=\Phi(\delta_{\tt})+\int_{0}^{s\wedge\tau_m}\mathbb{E}_{\delta_{\tt}}[\mathscr{L}\Phi(X_r)]dr\leq\Phi(\delta_{\tt})+\int_{0}^{s}\left(\lambda\mathbb{E}_{\delta_{\tt}}[\Phi(X_r)]+C|\tt|^2\right)dr.$$
    Let $m\rightarrow\infty$ and $X_s^m\rightarrow X_s$ almost surely. We apply Fatou's lemma on the left and use $s\le 1$:
    \begin{align*}\mathbb{E}_{\delta_{\tt}}[\Phi(X_s)]\leq\liminf_{m\rightarrow\infty}\mathbb{E}_{\delta_{\tt}}[\Phi(X_s^m)]&\leq\Phi(\delta_{\tt})+\int_{0}^{s}\left(\lambda\mathbb{E}_{\delta_{\tt}}[\Phi(X_r)]+C|\tt|^2\right)dr\\&\le\Phi(\delta_{\tt})+C|\tt|^2+\lambda\int_{0}^{s}\mathbb{E}_{\delta_{\tt}}[\Phi(X_r)]dr
    \end{align*}
    Omitting the change of constants, we use Gr\"{o}nwall's inequality:
    $$\mathbb{E}_{\delta_{\tt}}[\Phi(X_s)]\leq(\Phi(\delta_{\tt})+C|\tt|^2)e^{\lambda s}\leq C|\tt|^2e^{\lambda s}.$$
    So we have
    $$\sup_{\tt\in\TT}\frac{\mathbb{E}_{\delta_{\tt}}[\Phi(X_s)]}{|\tt|^2}\leq Ce^{\lambda s}$$
    which completes the proof.
   \end{proof}
   Now we are ready to prove Theorem~\ref{Thm:a.e.ConvergenceIntro}:
   \begin{proof}[Proof of Theorem~\ref{Thm:a.e.ConvergenceIntro}]
       It is classical that $(\M_s)_{s\ge 0}$ defined as 
      \begin{align*}
          \M_s:=e^{-\lambda s}\langle X_s,h\rangle
      \end{align*}
      is a non-negative martingale, hence converges almost surely to some random variable $W$. Then Proposition~\ref{Pro:ConditionOfa.e.ConvergenceTheorem} together with \cite[Theorem~3.1]{bansaye2025stronglawlargenumbers} implies \eqref{eq:a.e.ConvergenceOff}. Since $\lambda=\lab-\theta$, $\lambda>0$ \  implies $\lab>0$. Then Proposition~\ref{eigenelements} yields $\gamma>\theta$, which in turn implies the uniform lower bound $\inf_{\tt\in\TT}h(\tt)>0$. That is $\sup_{\tt\in\TT}1/h(\tt)<+\infty$.
   \end{proof}

\appendix
\section{THE VARIATION OF GR\"{O}NWALL'S FORMULA}
We present here the proof of Lemma~\ref{lem.Gronwallvariation}. 

\begin{proof}
We prove the lemma by considering the different cases of $C$ and $g(0)$. 

First deal with the case $C = 0, g(0) > 0 $.Define $ h(t) $ for $t\geq 0$ by $h(t):= g(0)e^{-ct}$. Notice that $ h $ is positive with $ h(0) = g(0) $, and for every $0\leq u\leq t$ we have
$$h(t) - h(u) = -\int_u^t c h(s) \, ds . $$
Consider the set $ S = \{ t \geq 0 : g(t) > h(t) \} $. If $ S = \emptyset $ then the lemma holds true. Assume by contradiction that $ S \neq \emptyset $. In this case, take $a>0$ such that $ g(a) > h(a) $. Consider now
$$ b := \inf\{0\leq a' < a : g(t) > h(t),\quad \forall t \in (a', a)\} . $$
By continuity of $ g $ and $ h $, we have $ g(b) = h(b) $. Thus $ b < a $ and for every
$t \in (b, a]$ we have $g(t) > h(t)$. Now let $ \phi(t) = g(b) -c \int_b^t g(s) \, ds $. \eqref{eq.Gronwallvarationcondition} yields for every $t\geq b$
$$g(t) \leq \phi(t). $$
In order to compare $\phi(t)$ and $h(t)$ for $t \in [b, a]$, taking the derivative of the ratio $\frac{\phi(t)}{ h(t)}$ we have
$$ \left( \frac{\phi}{h} \right)' = \frac{\phi' h - h' \phi}{h^2} = \frac{-c gh + c h \phi}{h^2} = \frac{c h (\phi - g)}{h^2} \geq 0. $$
This implies that $\frac{\phi(t)}{h(t)}$ is increasing for $t \in [b, a]$. Since $\frac{\phi(b)}{h(b)} = 1$, we can conclude that for every $t\in [b,a]$
$$\phi(t) \geq h(t).  $$
Thus
$$ \int_b^t g(s) \, ds \leq \int_b^t h(s) \, ds,\quad\forall t \in [b, a] \quad. $$
This contradicts the fact that for every $t\in (b,a]$ we have $g(t)>h(t)$.

Next, we turn to the case $ C = 0, g(0) \leq 0 $. Consider the set $ S = \{ t \geq 0:g(t) > 0 \} $. If $ S = \emptyset $ then the lemma holds true. Assume by contradiction that $ S \neq \emptyset $. Again, take an element $ a>0$ such that $ g(a) > 0 $. Consider now
$$ b = \inf\{0\leq a' < a : g(t) > 0 ,\quad\forall t \in (a', a) \ \} . $$
By continuity of $ g $ and the fact that $ g(0) \leq 0 $, we have $ g(b) = 0$ and 
$$g(t) > 0, \quad \forall t\in(b,a]. $$
Then choose an arbitrary $ t_0 \in (b, a] $. $g(s)>0$ for any $s\in(b,a]$ implies that
$$ g(t_0) \leq -c\int_b^{t_0} g(s) \, ds < 0. $$
This contradicts the assumption.

Finally, we turn to the case $ C \neq 0 $. 
Define $ \hat{g}:= g - C /c $. For $0\leq u\leq t$ we have
$$\hat{g}(t) - \hat{g}(u) = g(t) - g(u) \leq \int_u^t (- c g(s) + C)\,ds = -\int_u^t c\hat{g}(s) \, ds. $$
Thus $ \hat{g} $ satisfies the conditions of Case 1 or Case 2. Then as a consequence for every $t\geq 0$ we have
$$\hat{g}(t) \leq [\hat{g}(0)]_+ e^{-c t}. $$
The conclusion of the lemma follows by replacing $ \hat{g} $ by $ g - C / c $ in the above equation. In summary, \eqref{eq.Gronwallvariation} is proved to be true. 
\end{proof}

\textbf{Acknowledgement.} 
This research is supported by the National Natural Science Foundation
of China (Nos. 12595280, 12595284).

\bibliographystyle{abbrv}
\bibliography{RRTRef}

@book {AN2004,
	AUTHOR = {Athreya, K. B. and Ney, P. E.},
	TITLE = {Branching processes},
	NOTE = {Reprint of the 1972 original [Springer, New York; MR0373040]},
	PUBLISHER = {Dover Publications, Inc., Mineola, NY},
	YEAR = {2004},
	PAGES = {xii+287},
	ISBN = {0-486-43474-5},
	MRCLASS = {60-02 (60J10 60J27 60J45 60J80)},
	MRNUMBER = {2047480},
}

@book {Harris63,
    AUTHOR = {Harris, Theodore E.},
     TITLE = {The theory of branching processes},
    SERIES = {Die Grundlehren der mathematischen Wissenschaften},
    VOLUME = {Band 119},
 PUBLISHER = {Springer-Verlag, Berlin; Prentice Hall, Inc., Englewood
              Cliffs, NJ},
      YEAR = {1963},
     PAGES = {xiv+230},
   MRCLASS = {60.67},
  MRNUMBER = {163361},
MRREVIEWER = {P.\ A. P. Moran},
}

@article{bellin2024uniformattachmentfreezingscaling,
	AUTHOR = {Bellin, \'Etienne and Blanc-Renaudie, Arthur and Kammerer,
	Emmanuel and Kortchemski, Igor},
	TITLE = {Uniform attachment with freezing: scaling limits},
	JOURNAL = {Ann. Inst. Henri Poincar\'e{} Probab. Stat.},
	FJOURNAL = {Annales de l'Institut Henri Poincar\'e{} Probabilit\'es et
	Statistiques},
	VOLUME = {61},
	YEAR = {2025},
	NUMBER = {4},
	PAGES = {2679--2708},
	ISSN = {0246-0203,1778-7017},
	MRCLASS = {60D05 (60F05)},
	MRNUMBER = {4981995},
	DOI = {10.1214/24-aihp1497},
	URL = {https://doi.org/10.1214/24-aihp1497},
}

@article{bellin2023uniform,
	AUTHOR = {Bellin, \'Etienne and Blanc-Renaudie, Arthur and Kammerer,
	Emmanuel and Kortchemski, Igor},
	TITLE = {Uniform attachment with freezing},
	JOURNAL = {Ann. Appl. Probab.},
	FJOURNAL = {The Annals of Applied Probability},
	VOLUME = {35},
	YEAR = {2025},
	NUMBER = {4},
	PAGES = {2882--2922},
	ISSN = {1050-5164,2168-8737},
	MRCLASS = {60C05 (05C82 60F17)},
	MRNUMBER = {4945094},
	DOI = {10.1214/25-AAP2190},
	URL = {https://doi.org/10.1214/25-AAP2190},
}

@article{bansaye2021growth,
	AUTHOR = {Bansaye, Vincent and Gu, Chenlin and Yuan, Linglong},
	TITLE = {A growth-fragmentation-isolation process on random recursive
	trees and contact tracing},
	JOURNAL = {Ann. Appl. Probab.},
	FJOURNAL = {The Annals of Applied Probability},
	VOLUME = {33},
	YEAR = {2023},
	NUMBER = {6B},
	PAGES = {5233--5278},
	ISSN = {1050-5164,2168-8737},
	MRCLASS = {60J80 (60J27 60J85)},
	MRNUMBER = {4677733},
	MRREVIEWER = {J\'anos\ Engl\"ander},
	DOI = {10.1214/23-aap1947},
	URL = {https://tlink.lib.tsinghua.edu.cn:443/https/443/org/doi/yitlink/10.1214/23-aap1947},
}

@article{bansaye2019non,
	AUTHOR = {Bansaye, Vincent and Cloez, Bertrand and Gabriel, Pierre and
	Marguet, Aline},
	TITLE = {A non-conservative {H}arris ergodic theorem},
	JOURNAL = {J. Lond. Math. Soc. (2)},
	FJOURNAL = {Journal of the London Mathematical Society. Second Series},
	VOLUME = {106},
	YEAR = {2022},
	NUMBER = {3},
	PAGES = {2459--2510},
	ISSN = {0024-6107,1469-7750},
	MRCLASS = {47A35 (35B40 47D06 60J80 92D25)},
	MRNUMBER = {4498558},
	DOI = {10.1112/jlms.12639},
	URL = {https://tlink.lib.tsinghua.edu.cn:443/https/443/org/doi/yitlink/10.1112/jlms.12639},
}

@incollection{baur2014cutting,
	AUTHOR = {Baur, Erich and Bertoin, Jean},
	TITLE = {Cutting edges at random in large recursive trees},
	BOOKTITLE = {Stochastic analysis and applications 2014},
	SERIES = {Springer Proc. Math. Stat.},
	VOLUME = {100},
	PAGES = {51--76},
	PUBLISHER = {Springer, Cham},
	YEAR = {2014},
	MRCLASS = {60C05 (05C50 05C80 60J80 60K35)},
	MRNUMBER = {3332709},
	MRREVIEWER = {Alexander Iksanov},
	DOI = {10.1007/978-3-319-11292-3\_3},
	URL = {https://doi.org/10.1007/978-3-319-11292-3_3},
}

@article{meir1974cutting,
	AUTHOR = {Meir, A. and Moon, J. W.},
	TITLE = {Cutting down random trees},
	JOURNAL = {J. Austral. Math. Soc.},
	FJOURNAL = {J. Austral. Math. Soc.},
	VOLUME = {11},
	YEAR = {1970},
	PAGES = {313--324},
	MRCLASS = {05.45},
	MRNUMBER = {284370},
	MRREVIEWER = {E.\ M.\ Palmer},
}

@article{athreya1968some,
	AUTHOR = {Athreya, Krishna Balasundaram},
	TITLE = {Some results on multitype continuous time {M}arkov branching
	processes},
	JOURNAL = {Ann. Math. Statist.},
	FJOURNAL = {Annals of Mathematical Statistics},
	VOLUME = {39},
	YEAR = {1968},
	PAGES = {347--357},
	ISSN = {0003-4851},
	MRCLASS = {60.67},
	MRNUMBER = {221600},
	MRREVIEWER = {F. L. Spitzer},
	DOI = {10.1214/aoms/1177698395},
	URL = {https://doi.org/10.1214/aoms/1177698395},
}

@article{Barlow,
	title={A branching process with contact tracing},
	author={Barlow,Martin T},
	journal={Preprint 2020},
	pages={available via https://www.math.ubc.ca/~barlow/preprints/112-bpct5.pdf}
}

@article{Lambert,
	title={A mathematical assessment of the efficiency of quarantining and contact tracing in curbing the COVID-19 epidemic},
	author={Lambert, Amaury},
	journal={Preprint on MedArxiv (2020)}
}

@article{bertoin2022,
	AUTHOR = {Bertoin, Jean},
	TITLE = {A model for an epidemic with contact tracing and cluster
	isolation, and a detection paradox},
	JOURNAL = {J. Appl. Probab.},
	FJOURNAL = {Journal of Applied Probability},
	VOLUME = {60},
	YEAR = {2023},
	NUMBER = {3},
	PAGES = {1079--1095},
	ISSN = {0021-9002,1475-6072},
	MRCLASS = {60J80 (92D25)},
	MRNUMBER = {4624056},
	MRREVIEWER = {Eliane\ R.\ Rodrigues},
	DOI = {10.1017/jpr.2022.112},
	URL = {https://tlink.lib.tsinghua.edu.cn:443/https/443/org/doi/yitlink/10.1017/jpr.2022.112},
}

@article{kesten1966limit,
	AUTHOR = {Kesten, H. and Stigum, B. P.},
	TITLE = {A limit theorem for multidimensional {G}alton-{W}atson
	processes},
	JOURNAL = {Ann. Math. Statist.},
	FJOURNAL = {Annals of Mathematical Statistics},
	VOLUME = {37},
	YEAR = {1966},
	PAGES = {1211--1223},
	ISSN = {0003-4851},
	MRCLASS = {60.67},
	MRNUMBER = {198552},
	MRREVIEWER = {I. J. Good},
	DOI = {10.1214/aoms/1177699266},
	URL = {https://doi.org/10.1214/aoms/1177699266},
}

@misc{bansaye2025stronglawlargenumbers,
      title={On the strong law of large numbers and Llog L condition for supercritical general branching processes}, 
      author={Vincent Bansaye and Tresnia Berah and Bertrand Cloez},
      year={2025},
      eprint={2503.03324},
      archivePrefix={arXiv},
      primaryClass={math.PR},
      url={https://arxiv.org/abs/2503.03324}, 
}

@misc{andre2025sharpllogl,
      title={Sharp $L \log L$ condition for supercritical Galton-Watson processes with countable types}, 
      author={Mathilde André and Jean-Jil Duchamps},
      year={2025},
      eprint={2503.05575},
      archivePrefix={arXiv},
      primaryClass={math.PR},
      url={https://arxiv.org/abs/2503.05575}, 
}

@article {VJ62,
    AUTHOR = {Vere-Jones, D.},
     TITLE = {Geometric ergodicity in denumerable {M}arkov chains},
   JOURNAL = {Quart. J. Math. Oxford Ser. (2)},
  FJOURNAL = {The Quarterly Journal of Mathematics. Oxford. Second Series},
    VOLUME = {13},
      YEAR = {1962},
     PAGES = {7--28},
      ISSN = {0033-5606,1464-3847},
   MRCLASS = {60.65},
  MRNUMBER = {141160},
MRREVIEWER = {J.\ S.\ Griffin, Jr.},
       DOI = {10.1093/qmath/13.1.7},
       URL = {https://doi.org/10.1093/qmath/13.1.7},
}

@book {Seneta06,
    AUTHOR = {Seneta, E.},
     TITLE = {Non-negative matrices and {M}arkov chains},
    SERIES = {Springer Series in Statistics},
      NOTE = {Revised reprint of the second (1981) edition [Springer-Verlag,
              New York; MR0719544]},
 PUBLISHER = {Springer, New York},
      YEAR = {2006},
     PAGES = {xvi+287},
      ISBN = {978-0387-29765-1; 0-387-29765-0},
   MRCLASS = {60-02 (15A48 60J10)},
  MRNUMBER = {2209438},
}

@book {jagers75,
    AUTHOR = {Jagers, Peter},
     TITLE = {Branching processes with biological applications},
    SERIES = {Wiley Series in Probability and Mathematical
              Statistics---Applied Probability and Statistics},
 PUBLISHER = {Wiley-Interscience [John Wiley \& Sons], London-New
              York-Sydney},
      YEAR = {1975},
     PAGES = {xiii+268},
      ISBN = {0-471-43652-6},
   MRCLASS = {60J80 (60J85 92A15)},
  MRNUMBER = {488341},
MRREVIEWER = {Dean\ H.\ Fearn},
}

@article{pichugin2020evolution,
  title={Evolution of multicellular life cycles under costly fragmentation},
  author={Pichugin, Yuriy and Traulsen, Arne},
  journal={PLoS computational biology},
  volume={16},
  number={11},
  pages={e1008406},
  year={2020},
  publisher={Public Library of Science San Francisco, CA USA}
}

@article{pichugin2017fragmentation,
  title={Fragmentation modes and the evolution of life cycles},
  author={Pichugin, Yuriy and Pe{\~n}a, Jorge and Rainey, Paul B and Traulsen, Arne},
  journal={PLoS computational biology},
  volume={13},
  number={11},
  pages={e1005860},
  year={2017},
  publisher={Public Library of Science San Francisco, CA USA}
}
\end{document}